\documentclass[12pt]{amsart}

\usepackage{amssymb,amsmath}
\usepackage{paralist}

\usepackage{lmodern}
\usepackage{microtype}

\usepackage[a4paper, hmargin={2cm}]{geometry}

\usepackage[english]{babel}
\usepackage{csquotes}
\usepackage{xcolor}

\usepackage{etoolbox}
\newtoggle{usebiblatex}

\togglefalse{usebiblatex}

\usepackage[colorlinks=true, citecolor=blue, urlcolor=violet, breaklinks=true]{hyperref}

\newtheorem{theorem}{Theorem}[section]
\newtheorem{proposition}[theorem]{Proposition}
\newtheorem{lemma}[theorem]{Lemma}
\newtheorem{corollary}[theorem]{Corollary}

\theoremstyle{definition}
\newtheorem{definition}[theorem]{Definition}
\newtheorem{example}[theorem]{Example}
\newtheorem{remark}[theorem]{Remark}
\newtheorem*{notation}{Notation}

\DeclareMathOperator{\conv}{conv}
\DeclareMathOperator{\cone}{cone}
\DeclareMathOperator{\relint}{relint}
\DeclareMathOperator{\spann}{span}
\DeclareMathOperator{\supp}{supp}
\DeclareMathOperator{\sgn}{sgn}

\newcommand{\set}[1]{\{#1\}}
\newcommand{\with}{\ \vrule\ }

\newcommand{\N}{\mathbb{N}}
\newcommand{\R}{\mathbb{R}}

\newcommand{\xb}{\mathbf{x}}
\newcommand{\yb}{\mathbf{y}}
\newcommand{\wb}{\mathbf{w}}
\newcommand{\pb}{\mathbf{p}}

\newcommand{\cE}{\mathcal{E}}
\newcommand{\cA}{\mathcal{A}}
\newcommand{\cB}{\mathcal{B}}
\newcommand{\cS}{\mathcal{S}}

\newcommand{\CS}{C_\cS}

\newcommand{\re}{r_{\mathrm{e}}}
\newcommand{\ro}{r_{\mathrm{o}}}
\newcommand{\Po}[1]{P^{\mathrm{odd}}_{#1}}
\newcommand{\Pe}[1]{P^{\mathrm{even}}_{#1}}

\author{Lukas Katth\"an}
\author{Helen Naumann}
\author{Thorsten Theobald}

\address{Lukas Katth\"an, Helen Naumann, Thorsten Theobald:
	Goethe-Universit\"at, FB 12 -- Institut f\"ur Mathematik,
	Postfach 11 19 32, D--60054 Frankfurt am Main, Germany}

\email{\{katthaen, naumann, theobald\}@math.uni-frankfurt.de}

\title[]{A unified framework of SAGE and SONC polynomials and its duality theory}

\date{\today}

\begin{document}
\begin{abstract}
We introduce and study a cone which consists
of a class of generalized polynomial functions and which 
provides a common framework for recent non-negativity 
certificates of polynomials in sparse settings. Specifically, 
this $\mathcal{S}$-cone generalizes and unifies sums of
arithmetic-geometric mean exponentials (SAGE) and
sums of non-negative circuit polynomials (SONC).
We provide a comprehensive
characterization of the dual cone of the $\mathcal{S}$-cone,
which even for its specializations provides novel and projection-free descriptions.
As applications of this result, 
we give an exact characterization of the
extreme rays of the $\mathcal{S}$-cone and thus also of its
specializations, and we provide a subclass of functions for 
which non-negativity coincides with membership in the 
$\mathcal{S}$-cone.

Moreover, we derive from the duality theory
an approximation result of non-negative univariate polynomials
and show that a SONC analogue of Putinar's Positivstellensatz
does not exist even in the univariate case.
\end{abstract}

\maketitle

\section{Introduction}

In recent years, several interrelated approaches for non-negative
polynomials and for non-negative exponential sums have been proposed,
which are aimed at sparse settings.
In \cite{chandrasekaran-shah-2016}, Chandrasekaran and Shah proposed (in the language of exponential sums/signomials) to consider sums of polynomial functions $f:\R^n_+ \to \R$ of the form $\sum_{\alpha \in \cA} c_{\alpha} \xb^{\alpha}$ for a given set $\cA\subseteq\R^n$
such that at most one term has a negative coefficient.
Non-negativity of signomials can be characterized in terms of the arithmetic-geometric mean inequality, and deciding membership in the resulting cone (SAGE cone) can be formulated as a relative entropy program.
In \cite{iliman-dewolff-resmathsci}, Iliman and de Wolff proposed to consider sums of non-negative circuit polynomials on $\R^n$ (SONC polynomials). For a certain subclass of polynomials called ST-polynomials, deciding membership in the cone of SONC polynomials can be formulated in terms of the optimization subclass of geometric programs \cite{diw-2019}.
Murray, Chandrasekaran and Wierman \cite{mcw-2018} have shown that an adaption of the SAGE setting to $\R^n$ gives exactly the same cone of polynomials as the SONC cone. This yields a computationally tractable method to decide membership of 
arbitrary polynomials in the SONC cone using a relative entropy program. 

While many aspects of these classes of polynomials are connected with open questions and research efforts, they clearly exhibit some fundamental structural phenomena adapted to sparse settings.
For example, it was shown by Murray et al.\ \cite{mcw-2018} for the SAGE cone and by Wang \cite{wang-nonneg} for the SONC cone that every polynomial in those cones has a cancellation-free representation.
Generally, SAGE and SONC approaches can be combined with semidefinite approaches to polynomial optimization, see Karaca, Darivianakis et al.\ \cite{karaca-2017} or Averkov \cite{averkov-2019}.
Moreover, by \cite[Theorem 2.16]{averkov-2019} and its proof, the SONC cone
is second-order-cone representable (but the size of the second-order 
formulation from that work depends on the actual values of the support vectors).
For a practical algorithm to compute SONC bounds via second-order cone
representations, see Magron and Wang \cite{wang-magron-2020}.

The goal of the present paper is to provide a uniform framework which covers all these classes as well as some more general settings.
Since non-negativity of a polynomial function $f(x_1, \ldots, x_n)$ on $\R_+^n$ is equivalent to non-negativity
of $f(|x_1|, \ldots, |x_n|)$ on $\R^n$, we consider the more general functions $f: \R^n \to \R \cup \set{\infty}$ of the form
\begin{equation}\label{eq:func1}
	f(\xb) = \sum_{\alpha \in \cA} c_\alpha |\xb|^{\alpha} + \sum_{\beta\in \cB} d_\beta \xb^{\beta},
\end{equation}
with sets of exponents $\cA\subseteq \R^n$, $\cB\subseteq \N^n\setminus(2\N)^n$, which also capture the signomial functions.
Based on a subset of these functions, we define the \emph{$\cS$-cone} $\CS(\cA,\cB)$ which provides the common generalization
of the cones mentioned above, see Definition~\ref{de:scone}.
Its atomic functions are called \emph{AG functions}, which are functions of the form~\eqref{eq:func1} with strong support conditions.
The AG functions can be seen as a (non-polynomial) generalization of polynomials coming from the arithmetic-geometric inequality.
Building upon the earlier work of the second and the third author \cite{dnt-2018} on the dual SONC cone, a particular focus is the structure and the use of the dual viewpoint.

Non-negative polynomials and polynomial optimization are ubiquitous in applications, and sparsity is one of the central structural properties that provides potential for efficient computation.
Besides classical application in control theory and robotics (see, e.g., \cite{ahmadi-majumdar-2019,henrion-garulli-2005} and the references therein), let us list the more recent applications of non-negative polynomials and polynomial optimization in the optimal power flow problem \cite{josz-diss}, collision avoidance \cite{ahmadi-majumdar-2016} or shape-constrained regression \cite{hall-diss}.

\subsection*{Contributions}
1. We show that fundamental properties of the SAGE and/or the SONC cone
also hold in the more general context of the $\cS$-cone.
In particular, every $f \in \CS(\cA,\cB)$ can be decomposed
into a sum of non-negative AG functions whose supports are contained
in the support of $f$. See Proposition~\ref{pr:nocancel}, which
unifies and generalizes the results of \cite{mcw-2018} for the
SAGE cone and of \cite{wang-nonneg} for the SONC cone.

2. We provide a comprehensive characterization of the dual
cone of the $\cS$-cone, see~Theorem~\ref{thm:equiv}. In particular,
we provide projection-free
characterizations in terms of AG functions supported
on the particular class of reduced circuits.
The characterizations of the dual cone go far beyond the
characterizations of the dual SAGE cone 
from \cite{chandrasekaran-shah-2016}
and the dual SONC cone from \cite{dnt-2018}, where the dual cones
are described in terms of projections. Our proofs
provide a uniform tool set for handling the various types of
cones.

3. Based on the characterizations of the dual of the 
$\cS$-cone, we provide several applications of the duality
theory.

(a) We show that every sum $f$ of non-negative AG functions can be written
as a sum of non-negative circuit functions whose supports are contained
in the support of $f$.
This unifies and generalizes the results from \cite{mcw-2018}
for the SAGE cone and of \cite{wang-nonneg} for the 
SONC cone.

(b) We give an exact characterization of the extreme rays of the
$\cS$-cone. Even for the particular case of the SAGE cone, this
characterization substantially sharpens the necessary conditions
in \cite{mcw-2018}.

(c) We show that not even in the univariate case, SONC polynomials do 
allow Putinar-type representations. This counterexample strengthens
and simplifies the result of Dressler, Kurpisz and de Wolff 
\cite{DKdW}, who have provided a multivariate counterexample.

(d) We give a characterization of a wide class of non-negative AG 
functions with simplex Newton polytopes. Using the dual
$\cS$-cone, this result unifies and generalizes the results
from \cite{iliman-dewolff-resmathsci} and \cite{mcw-2018} and provides a simpler proof.

(e) As a final application of the dual $\cS$-cone, we show that 
non-negative univariate polynomials can be approximated by SONC
polynomials.

\smallskip

As further related work, let us mention the exploitation of 
sparsity and symmetries to derive specific SDP relaxations for
polynomial optimization \cite{kkw-2005,mcd-2017,rtjl-2013,
wml-2019,wml-2020,wlt-2018}.

\section{The \texorpdfstring{$\cS$}{S}-Cone} \label{se:scone}
In this section, we introduce AG functions and the $\cS$-cone.
We show that every non-negative function in the $\cS$-cone has a cancellation-free representation 
(see Proposition \ref{pr:nocancel}) and characterize non-negativity of an AG function in terms of the relative entropy function (see Theorem \ref{thm:oddImplication}).

\begin{notation}
	Throughout the article we use the notations $\N=\{0,1,2,3,\ldots\}$ and $\R_+=\{x\in\R:x\ge 0\}$.
	Moreover, for a finite subset $\cA\subseteq\R^n$, denote by $\R^\cA$ the set of $|\cA|$-dimensional vectors whose components are indexed by the set $\cA$.
\end{notation}

Our main object of study are functions $f: \R^n \to \R \cup \set{\infty}$ of the form
\begin{equation}\label{eq:generalfunction}
	f(\xb) = \sum_{\alpha \in \cA} c_\alpha |\xb|^{\alpha} + \sum_{\beta\in \cB} d_\beta \xb^{\beta},
\end{equation}
where $\cA\subseteq \R^n$, $\cB\subseteq \N^n\setminus(2\N)^n$ are finite sets of exponents,
$\{c_\alpha: \alpha\in \cA\},\{d_\beta:\beta\in\cB\}\subseteq \R$.
Here we use the notations
\[
|\xb|^{\alpha} = \prod_{j=1}^n |x_j|^{\alpha_{j}} \qquad\text{ and }\qquad \xb^{\beta} = \prod_{j=1}^n x_j^{\beta_{j}},
\]
and if one component of $\xb$ is zero and the corresponding exponent is negative, then we set $|\xb|^\alpha = \infty$.

For two finite sets $\emptyset \neq \cA \subseteq \R^n, \cB \subseteq \N^n\setminus(2\N)^n$, let
\[
\R[\cA, \cB] := \spann_\R(\{|\xb|^{\alpha} \with \alpha\in \cA\} \cup \{\xb^{\beta} \with \beta\in \cB\})
\]
denote the space of all functions of the form \eqref{eq:generalfunction} with given sets of exponents. 
This is a vector space of dimension $\dim \R[\cA,\cB]=|\cA|+|\cB|$.

\begin{remark} \label{rem:SONCcone}
		(1)\ \ If $\cA \subseteq (2\N)^n$, then $\R[\cA, \cB]$ is exactly the space of polynomials with exponent vectors in $\cA \cup \cB$.
		For this reason, we sometimes refer to elements of $\cA$ as \emph{even exponents} and to elements of $\cB$ as \emph{odd exponents}.
	\begin{asparaenum}\setcounter{enumi}{1}
		\item If $\cB = \emptyset$, then $\R[\cA, \cB]$ can be identified with the space of \emph{signomials}, i.e., functions of the form
		\[ \yb \mapsto \sum_{\alpha \in \cA} c_\alpha \exp(\alpha^T \yb) \]
		via the identification $|x_i| = \exp(y_i)$.
		\item It is no restriction to exclude sets in $(2 \N)^n$ from $\cB$, since for exponents $\beta \in (2 \N)^n$, we have $|\mathbf{x}|^{\beta} = \mathbf{x}^\beta$.
		\item $\cA$ and $\cB$ are not necessarily disjoint (cf.\ Example~\ref{ex:circuitfunc} below).
		\end{asparaenum}
\end{remark}

We study the non-negativity of functions in $\R[\cA, \cB]$ using the following building blocks:
\begin{definition}
	Let $f = \sum_{\alpha \in \cA} c_\alpha |\xb|^\alpha + \sum_{\beta \in \cB} d_\beta \xb^\beta$.
	We say that $f$ is
	\begin{enumerate}
		\item an \emph{even AG function} if at most one of the $c_\alpha$ is negative and all the $d_\beta$ are zero; and
		\item an \emph{odd AG function} if all the $c_\alpha$ are non-negative and at most one of the $d_\beta$ is nonzero.
	\end{enumerate} 
	$f$ is called an \emph{AG function} (\emph{arithmetic-geometric mean function}) 
	if $f$ is an even AG function or
	an odd AG function.
\end{definition}

Note that non-negative even AG functions correspond exactly to the 
AGE functions (arithmetic-geometric exponentials) studied in 
\cite{chandrasekaran-shah-2016} and \cite{mcw-2018}.

We arrive at the central definition of this section.
\begin{definition}[$\cS$-cone] \label{de:scone}
	Let $\emptyset \neq \cA \subseteq \R^n$, $\cB \subseteq \N^n\setminus(2\N)^n$ be finite sets.
	The \emph{$\cS$-cone} $\CS(\cA, \cB)$ is defined as
	\[ \CS(\cA, \cB) := \cone( f\in \R[\cA, \cB] \with f \text{ is a non-negative AG function}), 
	\]
	where $\cone$ denotes the conic hull.
\end{definition}

\begin{remark}\hspace*{0.1ex}\label{rem:SONCequals}
	\begin{asparaenum}
		\item If $\cB = \emptyset$, then the $\cS$-cone can be identified with the SAGE cone using the substitution in Remark \ref{rem:SONCcone}(2). Formally, for 
		finite $\cA\subseteq\R^n$, $\cA' \subsetneq \cA$ and 
		$\beta \in \cA \setminus \cA'$, we set
	    \[
	      C_{\mathrm{SAGE}}(\cA) = \sum_{\beta \in \cA} 
		  C_{\text{AGE}}(\cA \setminus \{\beta\},\beta),
		\]
		where for $\mathcal{A'} := \mathcal{A} \setminus \{\beta\}$
				\[
				C_{\mathrm{AGE}}(\cA',\beta) = \Big\{c\in\R^\cA: 
				c_{\alpha} \ge 0 \text{ for } \alpha \in \cA', \,
				\sum\limits_{\alpha\in\cA'} c_\alpha \exp(\alpha^Tx)
				+ c_\beta \exp(\beta^Tx)\ge 0
				\text{ on } \R^n\Big\}.
				\]
		\item 
		If $\cA \subseteq (2 \N)^n$, then $\CS(\cA,\cB)$ is the cone
		of SONC polynomials supported on $\cA \cup \cB$ from \cite{iliman-dewolff-resmathsci,averkov-2019}. In those papers, the SONC cone is defined in terms of circuit polynomials (see Remark~\ref{re:circuits1}). The equivalence of the definitions
		was established in \cite{mcw-2018} and also follows from our more general
		result in Proposition~\ref{prop:AGisscf}.
		\item An example where 
		the cone ${\CS}(\cA,\cB)$ is different from both the SAGE cone and the SONC cone is given by $\cA=\set{1,4}$ and $\cB=\set{3}$.
	\end{asparaenum}
\end{remark}

For a non-empty finite set 
$\cA\subseteq \R^n$ and $\beta\in \N^n \setminus(2\N)^n$ let
\[
\Po{\cA, \beta} := \left\{ f \with f = \sum_{\alpha \in \cA} c_\alpha |\xb|^\alpha + d \xb^{\beta}, f(\xb) \geq 0 \; \,  \forall\; \xb\in\R^n, c\in\R_+^\cA, d \in\R\right\}
\]
be the cone of non-negative odd AG functions supported on $(\cA, \beta)$, and similarly  for $\beta\in\R^n \setminus \cA$ let
\[
\Pe{\cA, \beta} := \left\{ f \with f = \sum_{\alpha \in \cA} c_\alpha |\xb|^\alpha + d |\xb|^{\beta}, f(\xb) \geq 0 \; \, \forall \;\xb\in\R^n, c\in\R_+^\cA, d \in\R\right\}
\]
be the cone of non-negative even AG functions supported on $(\cA, \beta)$.
Note that, by definition, 
\begin{align} \label{eq:decomp}
\CS(\cA, \cB) = \sum_{\alpha \in \cA} \Pe{\cA\setminus\set{\alpha}, \alpha} + \sum_{\beta \in \cB} \Po{\cA, \beta}.
\end{align}

As pointed out by a referee, \eqref{eq:decomp} implies the following alternative representation of the $\cS$-cone. 

\begin{proposition}\label{prop:SconeSAGE}
	Let $\emptyset\ne \cA\subseteq\R^n$, $\cB\subseteq\N^n\setminus(2\N)^n$ be
	finite and $\mathbf{e}_\alpha$ denote the unit vector in $\R^{\mathcal{A} \cup \mathcal{B}}$ indexed with $\alpha\in\cA\cup\cB$. Then,
	\begin{align*}
		\CS(\cA,\cB)= & \left\{\sum\limits_{\alpha\in\cA}c_\alpha|x|^\alpha + \sum\limits_{\beta\in\cB}d_\beta x^\beta\in\R[\cA,\cB]: \sum\limits_{\alpha\in\cA}c_\alpha \cdot \mathbf{e}_\alpha - \sum\limits_{\beta\in\cB}|d_\beta|\cdot \mathbf{e}_\beta\in C_{\mathrm{SAGE}}(\cA\cup\cB)\right\}\\
		= & \left\{\sum\limits_{\alpha\in\cA}c_\alpha|x|^\alpha + \sum\limits_{\beta\in\cB}d_\beta x^\beta\in\R[\cA,\cB]: \right. \\ 
& \: \left.
\exists t\in\R^\cB, \sum\limits_{\alpha\in\cA}c_\alpha \cdot \mathbf{e}_\alpha + \sum\limits_{\beta\in\cB}t_\beta\cdot \mathbf{e}_\alpha\in C_{\mathrm{SAGE}}(\cA\cup\cB), \,
		t_\beta\le -|d_\beta| \text{ for all }\beta\in\cB\right\}.
	\end{align*}
\end{proposition}
In the case $\mathcal{A} \cap \mathcal{B} = \emptyset$, we can shortly write
$\sum\limits_{\alpha\in\cA}c_\alpha \cdot \mathbf{e}_\alpha - \sum\limits_{\beta\in\cB}|d_\beta|\cdot \mathbf{e}_\beta = (c,-|d|)$, where $|d|$ denotes the component-wise absolute value.
If there exists some $\beta \in \mathcal{A} \cap \mathcal{B}$, 
then the corresponding coefficient in 
the SAGE cone $c_\beta-|d_\beta|$ appears only once in the set $\R^{\cA\cup\cB}$.
However, by slight abuse of notation, we also write
$\sum\limits_{\alpha\in\cA}c_\alpha \cdot \mathbf{e}_\alpha
-\sum\limits_{\beta\in\cB}|d_\beta|\cdot \mathbf{e}_\beta$ shortly as $(c,-|d|)$.

\begin{proof}
	If $f=\sum\limits_{\alpha\in\cA}c_\alpha |x|^\alpha + \sum\limits_{\beta\in\cB}d_\beta x^\beta \in \CS(\cA,\cB)$, then~\eqref{eq:decomp} gives a decomposition 
		\begin{align*}
			f=\sum\limits_{\alpha\in\cA}f_\alpha^{\text{even}} + \sum\limits_{\beta\in\cB}f_\beta^{\text{odd}} 
		\end{align*}
		with $f_\alpha^{\text{even}}\in \Pe{\cA\setminus\set{\alpha},\alpha}$ for all $\alpha\in\cA$ and $f_\beta^{\text{odd}}=\sum\limits_{\alpha\in\cA}c_\alpha^{(\beta)}|x|^\alpha+ d_\beta x^\beta\in \Po{\cA, \beta}$ for every $\beta\in\cB$. Defining the functions
		\begin{align*}
		\tilde{f}_\alpha^{\text{even}}& \ = \ f_\alpha^{\text{even}} \text{ for all }\alpha\in\cA\\
		\text{and }	\tilde{f}_\beta^{\text{even}}& \ = \ \sum\limits_{\alpha\in\cA}c_\alpha^{(\beta)}|x|^\alpha- |d_\beta| |x|^\beta = f_\beta^{\text{odd}} - d_\beta x^\beta - |d_\beta| |x|^\beta \text{ for all }\beta\in\cB,
		\end{align*}
symmetry implies $\tilde{f}_\beta^{\text{even}}\in \Pe{\cA,\beta}$ and hence, $\tilde{f}=\sum\limits_{\alpha\in\cA}\tilde{f}_\alpha^{\text{even}} + \sum\limits_{\beta\in\cB}\tilde{f}_\beta^{\text{even}} \in \CS(\cA\cup\cB, \emptyset)$. 
Remark \ref{rem:SONCequals}(1) then shows that the coefficient vector of $\tilde{f}$
is contained in $C_{\mathrm{SAGE}}(\mathcal{A} \cup \mathcal{B})$.
	
The converse direction of the first equation
follows immediately with the substitution in Remark \ref{rem:SONCcone}(2). 

The second equation, which exhibits the convexity of the $\mathcal{S}$-cone,
is an immediate consequence of the first one.
\end{proof}

In our definition of the $\cS$-cone, we exclude sums of non-negative AG functions with support $\cA\cup\cB$ for $\cA\subseteq \R^n,\cB\subseteq\N^n\setminus(2\N)^n$, where the corresponding AG functions have bigger support than $\cA\cup \cB$.
This could happen, for example, if two summands cancel in the sum.
For a better understanding of the problem, we have a look at the following example.

\begin{example}
	Let $\cA:=\{\frac{1}{3},\frac{7}{3}\},\cB:=\{1\}$. Consider the two non-negative AG functions
	\begin{align*}
		& f_1:=|x|^{\frac{1}{3}}+x+x^2,\\
		& f_2:=|x|^{\frac{1}{3}}-x^2+|x|^{\frac{7}{3}},
	\end{align*}
	whose support is not contained in $\cA\cup\cB$. But the sum
	\begin{align*}
		f:=f_1+f_2=2|x|^{\frac{1}{3}}+x+|x|^{\frac{7}{3}}
	\end{align*}
	is itself a non-negative AG function, whose support is contained in $\cA\cup \cB$.
\end{example}
In fact, this restriction is not really a restriction. The following proposition states
that every sum $f$ of non-negative AG functions whose support is bigger than the support of the sum can be decomposed into a sum of non-negative AG functions 
whose supports are contained in the support of $f$.

For the SAGE case, this was already proven in \cite[Theorem 2]{mcw-2018} and for the SONC case this follows from the more detailed result of \cite{wang-nonneg}.

\begin{proposition} \label{pr:nocancel}
	Let $\emptyset \neq\cA\subseteq \R^n, \cB\subseteq \N^n\setminus(2\N)^n$ be finite sets and $f \in \R[\cA, \cB]$.
	If $f\in \CS(\cA',\cB')$ for some $\cA'\supseteq \cA$, $\N^n\setminus(2\N)^n\supseteq \cB' \supseteq \cB$, then $f\in \CS(\cA,\cB)$ as well.
	Equivalently, it holds that
	\[ \CS(\cA, \cB) = \CS(\cA', \cB') \cap \R[\cA, \cB].\]
\end{proposition}

\begin{proof}
By Proposition \ref{prop:SconeSAGE},
\begin{align*}
	\CS(\cA,\cB)= \left\{\sum\limits_{\alpha\in\cA}c_\alpha|x|^\alpha + \sum\limits_{\beta\in\cB}d_\beta x^\beta\in\R[\cA,\cB]: \sum\limits_{\alpha\in\cA}c_\alpha \cdot \mathbf{e}_\alpha - \sum\limits_{\beta\in\cB}|d_\beta|\cdot \mathbf{e}_\beta\in C_{\mathrm{SAGE}}(\cA\cup\cB)\right\},
\end{align*}
with $\mathbf{e}_\alpha$ denoting the unit vector with respect to $\alpha$ for $\alpha\in\cA$, resp. $\cB$. Let $f\in\CS(\cA',\cB')\cap \R[\cA,\cB]$ with coefficient vector $(c,d)$ and hence $(c,-|d|)\in C_{\mathrm{SAGE}}(\cA'\cup\cB')\cap \R^{\cA\cup\cB}$ where the absolute value is component-wise. The already mentioned statement for the SAGE-case (\cite{mcw-2018}, Theorem 2) states that $C_{\mathrm{SAGE}}(\cA\cup\cB)= C_{\mathrm{SAGE}}(\cA'\cup\cB')\cap\R^{\cA\cup\cB}$. Hence, $(c,-|d|)\in C_{\mathrm{SAGE}}(\cA\cup\cB)$ as well as 
\begin{align*}
	f\in & \left\{\sum\limits_{\alpha\in\cA}c_\alpha|x|^\alpha + \sum\limits_{\beta\in\cB}d_\beta x^\beta\in\R[\cA,\cB]: \sum\limits_{\alpha\in\cA}c_\alpha \cdot \mathbf{e}_\alpha - \sum\limits_{\beta\in\cB}|d_\beta|\cdot \mathbf{e}_\beta\in C_{\mathrm{SAGE}}(\cA\cup\cB)\right\}\\
	= & \ \CS(\cA, \cB),
\end{align*}
again by Proposition \ref{prop:SconeSAGE}. The other inclusion is obvious.
\end{proof}

Our next result characterizes non-negative AG functions.
It is a slight generalization of \cite[Lemma~2.2]{chandrasekaran-shah-rel-entropy} to the setting of AG functions.
The following notation is useful to state the theorem:

\begin{notation}
	For a non-empty finite set $\cA \subseteq \R^n$ and $\beta \in \R^n$, let $\Lambda(\cA,\beta)$ be the polytope
	\begin{equation}\label{eq:capitallambda}
		\Lambda(\cA, \beta) := \left\{ \lambda \in \R^\cA_+ \with \sum_{\alpha\in \cA} \lambda_\alpha \alpha = \beta,\ \sum_{\alpha\in\cA}\lambda_{\alpha} = 1\right\}.
	\end{equation}
	Note that $\Lambda(\cA, \beta) \neq \emptyset$ if and only if $\beta$ is contained in the convex hull of $\cA$.
	In the special case that $\cA$ is affinely independent, $\Lambda(\cA, \beta)$ consists of a single element, which we denote by $\lambda(\cA, \beta)$.
\end{notation}

Let $\cA \subseteq \R^n$ be a non-empty finite set.
We denote by $D:\R_{>0}^\cA \times \R_{>0}^\cA \to \R$,
\[
D(\nu, \gamma) \ = \ \sum_{\alpha \in \cA} \nu_\alpha \ln \left( \frac{\nu_\alpha}{\gamma_\alpha} \right), \quad \nu, \gamma \in \R_{>0}^\cA
\]
the \emph{relative entropy function}. It can be extended to $\R_{+}^\cA \times \R_{+}^\cA \to \R \cup \set{\infty}$ using the
usual conventions $0 \cdot \ln \frac{0}{y} = 0$ for $y \geq 0$ and $y \cdot \ln \frac{y}{0} = \infty$ for $y > 0$.

\begin{theorem}\label{thm:oddImplication}
	Let $\cA\subseteq \R^n$ be a non-empty finite set and $f$
	be an AG function of the form
	\[
	f = \sum_{\alpha\in \cA} c_\alpha |\xb|^\alpha + 
	\begin{cases}
	d |\xb|^\beta
	\text{ with } \beta \in \R^n \setminus\cA & \text{ if $f$ is even,} \\
	d \xb^\beta 
	\text{ with } \beta \in \N^n \setminus(2\N)^n & 
	\text{ if $f$ is odd,}
	\end{cases}
	\]
        where $c_{\alpha} \ge 0$ for all 
        $\alpha \in \mathcal{A}$ and $d \in \R$.
	Then the following statements are equivalent:
	\begin{enumerate}
		\item $f(\xb) \geq 0$ for all $\xb \in \R^n$.
		\item There exists a $\nu\in\R_{+}^{\cA}$ such that
		$\sum_{\alpha\in \cA} \nu_{\alpha}\alpha=(\sum_{\alpha\in \cA} \nu_{\alpha})\beta$
		and 
		\[
		D(\nu, e\cdot c) \le
		\begin{cases}
		d & \text{ if $f$ even,} \\
		-|d| & \text{ if $f$ odd.}
		\end{cases}
		\]
		\item There exists a $\lambda \in \Lambda(\cA, \beta)$ such that
		\[
		\prod_{\alpha \in \cA} \left(\frac{c_\alpha}{\lambda_\alpha}\right)^{\lambda_\alpha} \ge 
		\begin{cases}
		-d & \text{ if $f$ even,} \\
		|d| & \text{ if $f$ odd.}
		\end{cases}
		\]
	\end{enumerate}
\end{theorem}
A vector $\lambda \in  \Lambda(\cA, \beta)$ as in this theorem is called 
an \emph{AG witness}.

\begin{proof}[Proof of Theorem \ref{thm:oddImplication}]
	Before we prove this theorem, observe that for the AGE cone $C_{\text{AGE}}(\cA,\beta)$, defined in Remark \ref{rem:SONCequals}, we have  $(c,d)\in C_{\text{AGE}}(\cA,\beta)$ if and only if there exists $\nu\in\R_+^\cA$ such that $\sum_{\alpha\in \cA} \nu_{\alpha}\alpha=(\sum_{\alpha\in \cA} \nu_{\alpha})\beta$ and 
	\begin{align*}
		D(\nu,e\cdot c) \le d
	\end{align*}
	(compare \cite{chandrasekaran-shah-2016}, Section $2.1$). If $f$ is an even AG-function, this is exactly the equivalence $(1)\Leftrightarrow (2)$ due to Remark \ref{rem:SONCequals}. If $f$ is an odd AG-function, the equivalence follows from Proposition \ref{prop:SconeSAGE} and the mentioned observation.
	
	For the implication (2) $\implies$ (3), set 
	$\lambda := (\sum_{\alpha\in \cA} \nu_{\alpha})^{-1}\nu$.
	It is clear from the properties of $\nu$ that $\lambda \in \Lambda(\cA, \beta)$.
	The discussion in \cite[p.~1151]{chandrasekaran-shah-2016} shows that 
	\[ \prod_{\alpha \in \cA}\left(\frac{c_\alpha}{\lambda_\alpha}\right)^{\lambda_\alpha} \geq - D(\nu, e\cdot c) \]
	and thus this $\lambda$ has the desired properties.
	The implication (3) $\implies$ (1) is a direct consequence of the weighted
	arithmetic-geometric mean inequality:
	\begin{align*}
		\sum\limits_{\alpha\in\cA} c_\alpha|x|^\alpha
		\overset{\text{AM/GM-inequality}}{\ge } \prod\limits_{\alpha\in\cA} \left(\frac{c_\alpha}{\lambda_\alpha}|x|^\alpha\right)^{\lambda_\alpha}= \prod\limits_{\alpha\in\cA} \left(\frac{c_\alpha}{\lambda_\alpha}\right)^{\lambda_\alpha}|x|^\beta.
	\end{align*}
	Using $(3)$, we obtain
	\begin{align*}
\sum\limits_{\alpha\in\cA} c_\alpha|x|^\alpha	+ \begin{cases}
		d|x|^\beta\\ dx^\beta
	\end{cases}\ge |x|^\beta \begin{cases}
	-d+d\\ |d|-\sgn(x)\cdot d
	\end{cases}\ge 0.
\end{align*}	
	 As we  already know that $(1)\Leftrightarrow (2)$, we obtain the desired statement.
\end{proof}

\begin{example}\label{ex:circuitfunc}
	Let $\cA = \cB = \set{1} \subseteq \N$.
	A typical AG function with this support is
	\[ g(x) = c_1 |x| + c_2 x. \]
	Since the equality condition in statement $(2)$ of Theorem~\ref{thm:oddImplication}
	is trivially satisfied, we have $g(x) \geq 0$ for all $x \in \R$ if and only if there exists a $\nu \in \R_{+}$ with
	\begin{equation} 
	  \label{eq:xx}
	    \nu \ln\left(\frac{\nu}{e c_1}\right) \leq -|c_2|. 
	\end{equation}
	If $\nu\ge 0$, the latter condition can be simplified to $|c_2| \leq c_1$. For the case $\nu=0$, this is clear from our setting $0 \cdot \ln 0 = 0$, and to see it for
	$\nu>0$, rewrite~\eqref{eq:xx} as 
	\begin{align*}
	c_1\left(\frac{\nu}{c_1}\right)\ln\left(\frac{\nu}{e c_1}\right)\le -|c_2|.
	\end{align*}
	Since the function $x\ln\left(\frac{x}{e}\right)$ attains its minimum at $x=1$ (which means $x\ln\left(\frac{x}{e}\right)\ge-1$), we obtain the claimed result. It is in particular the one of statement (3) in Theorem~\ref{thm:oddImplication}.
\end{example}

For later use, we note that our cones of interest are closed:
\begin{proposition}\label{prop:closed}
	The cones $\Po{\cA, \beta}, \Pe{\cA, \beta}$ and $\CS(\cA, \cB)$ are closed pointed convex cones.
\end{proposition}
\begin{proof}
	It is clear that all three cones are pointed, since the only non-negative function $f$ where $-f$ is non-negative as well is the zero function.
	The cones $\Po{\cA, \beta}$ and $\Pe{\cA, \beta}$ are defined as (infinite) intersections of closed halfspaces, and thus they are closed.
	Finally, since finite sums of closed pointed convex cones are again closed, the cone $\CS(\cA, \cB)$ is closed as well.
\end{proof}

\section{Circuits and the dual of the \texorpdfstring{$\cS$}{S}-cone\label{se:circuits-dual}}
In this section, we introduce circuit functions and
provide several characterizations of the dual $\cS$-cone
(see Theorem~\ref{thm:equiv}).

We can identify the dual space of $\R[\cA,\cB]$ with $\R^{(\cA,\cB)} := \R^{\cA} \times \R^{\cB}$.
For $f \in \R[\cA,\cB]$ with coefficients $(c_\alpha)_{\alpha\in\cA},(d_\beta)_{\beta \in \cB}$ and an element $(v,w) \in \R^{(\cA,\cB)}$, we consider the natural pairing
\begin{equation}\label{eq:pairing}
	(v,w)(f) \ = \ \sum\limits_{\alpha\in \cA}v_\alpha c_\alpha +\sum\limits_{\beta \in \cB} w_\beta d_\beta \, .
\end{equation}
Using this notation, the dual cone $\CS({\cA,\cB})^*$ is defined as
\[
	\CS({\cA,\cB})^* \ = \ \{ (v,w) \in \R^{(\cA,\cB)} \mid (v,w)(f) \ge 0 \text{ for all } f \in \CS({\cA,\cB}) \}.
\]

Now we consider the representation of AG functions in terms
of circuit functions. Here, \emph{$\relint$} and \emph{$\conv$} denote the 
relative interior and the convex hull of a set.

\begin{definition}\label{de:circuits}
	A \emph{circuit} is a pair $(A, \beta)$, where $A \subseteq \R^n$ is affinely independent and $\beta \in \relint\conv(A)$.
	For finite sets $\cA, \cB \subseteq \R^n$, let
	\begin{align*} 
		I(\cA,\cB) &:= \set{ (A,\beta) \text{ circuit} \with A\subseteq \cA, \beta \in \cB}
	\end{align*}
	denote the set of all circuits on $\cA, \cB$.
	In particular, for $\cA\subseteq \R^n,\cB\subseteq\N^n\setminus(2\N)^n$ we call $I(\cA,\cA)$ the set of all \emph{even} circuits and $I(\cA,\cB)$ the set of all \emph{odd} circuits.
\end{definition}

\begin{definition}
	Let $(A, \beta)$ be a circuit.
	\begin{enumerate}
	\item An \emph{even circuit function} supported on $(A, \beta)$ is an AG function of the form
	\[ f = \sum_{\alpha \in A} c_\alpha |\xb|^\alpha + d |\xb|^\beta\]
	with $c_\alpha >0$ for all $\alpha\in A$ and $d\in\R$.
	\item For $\beta \in \N^n \setminus (2\N)^n$, an \emph{odd circuit function} supported on $(A, \beta)$ is an AG function of the form
		\[ f = \sum_{\alpha \in A} c_\alpha |\xb|^\alpha + d \xb^\beta\]
		with $c_\alpha >0$ for all $\alpha\in A$ and $d\in\R$.
	\end{enumerate}
	We call $\beta$ the \emph{inner} exponent of $f$ and the other exponents are the \emph{outer} exponents.
\end{definition}

\begin{remark}\label{re:circuits1}
		(1) In case of a circuit, the vector $\lambda \in \Lambda(\cA, \beta)$ in Theorem \ref{thm:oddImplication} is unique, and thus the non-negativity of $f$ can be expressed in terms of the \emph{circuit number}
		\begin{equation}\label{eq:circuitnumber}
			\Theta_f =\prod\limits_{\lambda_\alpha\neq0}\left(\frac{c_\alpha}{\lambda_\alpha}\right)^{\lambda_\alpha}, 
		\end{equation}
		which was introduced in \cite{iliman-dewolff-resmathsci}.

		(2) Iliman and de Wolff also introduced the notion of 
		\emph{circuit polynomials} in \cite{iliman-dewolff-resmathsci}. Every  circuit polynomial is an even circuit function if the inner exponent is even and an odd circuit function if the inner exponent is odd. With this, circuit polynomials form a special case of circuit functions.
\end{remark}

Next, we introduce \emph{reduced} circuits, which will be used in
Section~\ref{sec:extremerays} to determine the extreme rays of the $\cS$-cone. 
\begin{definition}\label{de:reducedcircuit}
	For a circuit $(A,\beta)$ let 
	\begin{align*}
	\re(A,\beta) &:=|\left(\conv(A) \setminus (A \cup\set{\beta})\right)\cap \cA|\ \text{ and}\\
	\ro(A,\beta) &:=|\left(\conv(A) \setminus A\right) \cap \cA|.
	\end{align*}
	An even circuit $(A, \beta)$ is called \emph{reduced} if $\re(A,\beta)=0$ and an odd circuit $(A, \beta)$ is called \emph{reduced} if $\ro(A,\beta)=0$.
\end{definition}

In other words, reduced circuits contain no elements of $\cA$ in their convex hull except those which are trivially there.
Note that for $\beta \in \cA \cap \cB$, it is possible that a circuit is reduced as an even circuit, but not reduced as an odd circuit.
See Example~\ref{ex:extreme1} below.

We can now provide the following characterization of the dual $\cS$-cone $\CS(\cA,\cB)^*$.
Here, recall the definition of $\Lambda(\cA, \beta)$ from~\eqref{eq:capitallambda}
and that $\lambda(\cA, \beta)$ denotes the single element of $\Lambda(\cA, \beta)$ in the case of a circuit.
We use the convention that $0\ln(0) = 0$ and $\ln(0) = -\infty$.

\begin{theorem}\label{thm:equiv}
	Let $\emptyset \neq \cA \subseteq \R^n$ and $\cB \subseteq \N^n \setminus(2\N)^n$ be finite sets and let $(v,w) \in \R^{(\cA, \cB)}$.
	\begin{enumerate}[(1)]
		\item If $(v,w) \in \CS(\cA, \cB)^*$, then $v_\alpha \geq 0$ for all $\alpha \in \cA$.
		\item If the condition of part (1) is satisfied, then the following are equivalent:
		\begin{enumerate}[(a)]
			\item $(v,w)$ lies in the dual cone $\CS(\cA, \cB)^*$.
			\item For all $\beta \in \cA$ (respectively $\beta \in \cB$) and all $\lambda \in \Lambda(\cA, \beta)$, it holds that
			\[ \ln|v_\beta| \leq \sum_{\alpha\in \cA} \lambda_{\alpha} \ln(v_\alpha) \quad
			 (\text{respectively } \ln|w_\beta| \leq \sum_{\alpha\in \cA} \lambda_{\alpha} \ln(v_\alpha)). \]
			\item For every even circuit $(A, \beta) \in I(\cA,  \cA)$
			(respectively odd circuit $(A, \beta) \in I(\cA, \cB)$)
			 and $\lambda = \lambda(A, \beta)$, it holds that
			\[ \ln|v_\beta| \leq \sum_{\alpha\in A} \lambda_{\alpha} \ln(v_\alpha)  \quad
			(\text{respectively } \ln|w_\beta| \leq \sum_{\alpha\in A} \lambda_{\alpha} \ln(v_\alpha)). \]
			\item For every reduced even circuit $(A, \beta) \in I(\cA,  \cA)$ (respectively 
			reduced odd circuit $(A, \beta) \in I(\cA, \cB)$) 
			and $\lambda = \lambda(A, \beta)$, it holds that
			\[ \ln|v_\beta| \leq \sum_{\alpha\in A} \lambda_{\alpha} \ln(v_\alpha) \quad
			(\text{respectively }
			 \ln|w_\beta| \leq \sum_{\alpha\in A} \lambda_{\alpha} \ln(v_\alpha)). \]
		\end{enumerate}
	\end{enumerate}
\end{theorem}

Before we prove Theorem~\ref{thm:equiv} we consider the duals of the sub-cones $\Po{\cA, \beta}$ and $\Pe{\cA, \beta}$ of $\CS({\cA,\cB})$.

\begin{lemma}\label{lem:dualcone}
	Let $\cA\subseteq\R^n$ be a non-empty finite set.
	\begin{enumerate}[(1)]
		\item For $ \beta \in \N^n \setminus (2 \N)^n$, the dual cone of $\Po{\cA, \beta}$ consists of those $(v,w) \in \R^{(\cA ,\set{\beta})}$ where
		\begin{enumerate}[(a)]
			\item $v_\alpha \geq 0$ for all $\alpha \in \cA$, and
			\item $\ln|w_\beta| \leq \sum_{\alpha\in \cA} \lambda_{\alpha} \ln(v_\alpha)$
			for all $\lambda \in \Lambda(\cA, \beta)$.
		\end{enumerate}
		\item For $\beta \in \R^n \setminus \cA$, the dual cone of $\Pe{\cA, \beta}$ consists of those $(v,w) \in \R^{(\cA,\set{\beta})}$ satisfying (a), (b) and in addition
		\begin{enumerate}[(a)]
			\item[(c)] $w_\beta \geq 0$.
		\end{enumerate}
	\end{enumerate}
\end{lemma}

\begin{proof}
	We prove the even and odd case simultaneously.
	Let $(v,w) \in (\Po{\cA, \beta})^*$ or $(v,w) \in (\Pe{\cA, \beta})^*$.
	First we show that it satisfies the claimed conditions.
	\begin{asparadesc}
	\item[\normalfont{\emph{(a) and (c):}}]
		For every $\alpha \in \cA$, it holds that $|\xb|^\alpha \in \Po{\cA,\beta}$ resp. $|\xb|^\alpha \in \Pe{\cA,\beta}$ and thus $0 \leq (v,w)( |\xb|^\alpha) = v_\alpha$, as claimed.
		In the even case, we also have that $|\xb|^\beta \in \Pe{\cA,\beta}$ and thus by the same argument \emph{(c)} holds.
	\item[\normalfont{\emph{(b):}}] 
		Fix a  $\lambda \in \Lambda(\cA, \beta)$.
		First assume that $v_{\alpha} \neq 0$ for all $\alpha\in \cA$.
		Then
		\[ 
			f := \sum_{\alpha\in \cA} \left(\prod_{\alpha'\in \cA} v_{\alpha'}^{\lambda_{\alpha'}}\right)
			\frac{\lambda_\alpha}{v_\alpha} |\xb|^{\alpha} - 
			\begin{cases}
				|\xb|^\beta &\text{ in the even case},\\
				\sgn(w_\beta) \xb^\beta &\text{ in the odd case}
			\end{cases}
		\]
		is an (even or odd) AG function and a straightforward computation shows that $f$ satisfies the condition $(3)$ of Theorem~\ref{thm:oddImplication} (with the given $\lambda$), hence $f$ is non-negative.
		Thus,
		\[
		0 \leq (v,w)(f) = 
		\begin{cases}
		  \prod_{\alpha\in \cA} v_{\alpha}^{\lambda_\alpha} - w_\beta &
		  \text{ in the even case}, \\
		  \prod_{\alpha\in \cA} v_{\alpha}^{\lambda_\alpha} - \sgn(w_\beta) w_\beta &
		  \text{ in the odd case},
		  \end{cases}
		\]
		which is equivalent to property \emph{(b)}.
		Since the mapping~\eqref{eq:pairing} is continuous in $(v,w)$, the statements
		also hold if $v_{\alpha} = 0$ for some $\alpha \in \cA$.
	\end{asparadesc}
	
	For the converse implication,
	assume that $v$ satisfies conditions \emph{(a)}, \emph{(b)}, and in the even case also \emph{(c)}.
	
	We need to show that every non-negative AG function $f = \sum_{\alpha\in \cA} c_{\alpha} |\xb|^{\alpha} + d \xb^\beta$ resp. $f = \sum_{\alpha\in \cA} c_{\alpha} |\xb|^{\alpha} + d |\xb|^\beta$ satisfies $(v,w)(f) \geq 0$.
	Let $\lambda \in \Lambda(\cA, \beta)$ be an AG witness for $f$ as in Theorem~\ref{thm:oddImplication}.
	Observe that
	\[
		\sum_{\alpha\in \cA} v_{\alpha} c_{\alpha} 
		= \sum_{\alpha\in \cA} \lambda_{\alpha} \left( \frac{v_{\alpha} c_{\alpha}} {\lambda_{\alpha}} \right) 
		\geq \prod_{\alpha\in \cA} \left(\frac{v_{\alpha} c_\alpha}{\lambda_\alpha}\right)^{\lambda_\alpha}
		= \prod_{\alpha\in \cA} v_{\alpha}^{\lambda_\alpha} \cdot \prod_{\alpha\in \cA} \left(\frac{c_\alpha}{\lambda_\alpha}\right)^{\lambda_\alpha}
		\geq |w_\beta| \prod_{\alpha\in \cA} \left(\frac{c_\alpha}{\lambda_\alpha}\right)^{\lambda_\alpha},
	\]
	which implies
	\begin{equation}
	  \label{eq:nonneg1}
		(v,w)(f) = \sum_{\alpha\in \cA} v_{\alpha} c_{\alpha} + w_\beta d \geq
		|w_\beta|\left( \prod_{\alpha\in \cA} \left(\frac{c_{\alpha}}{\lambda_{\alpha}}\right)^{\lambda_{\alpha}} + \sgn(w_\beta) d\right).
	\end{equation}
	In the even case, we have $\sgn(w_\beta) = +1$ by \emph{(c)}, and the right expression in~\eqref{eq:nonneg1} is non-negative, because $f$ is a non-negative AG function.
	In the odd case, observe that then the non-negativity of $f$ yields non-negativity of 
	the right expression in~\eqref{eq:nonneg1} as well.
\end{proof}

\begin{remark}
	In the beginning of this section, we identified the dual space of $\R[\cA,\cB]$ with $\R^{(\cA,\cB)}$.
	Using the reverse identification and associating for every $v\in \R^{(\cA,\cB)}$ a function of form \eqref{eq:func1}, we can identify the dual cones $(\Pe{A,\beta})^*$ resp.~$(\Po{A,\beta})^*$ with the cones of all functions of the form \eqref{eq:func1} with coefficients in $(\Pe{A,\beta})^*$ resp.~$(\Po{A,\beta})^*$.
	If $\beta \in \conv(\cA)$, then by Theorem~\ref{thm:oddImplication} and Lemma~\ref{lem:dualcone}, it is easy to see that $(\Pe{A,\beta})^* \subseteq \Pe{A,\beta}$ and $(\Po{A,\beta})^* \subseteq \Po{A,\beta}$.
	
	In particular, this means that every function of the form \eqref{eq:func1} with coefficients in $(\Pe{A,\beta})^*$ resp. $(\Po{A,\beta})^*$ is non-negative.
	Hence, $\CS(\cA,\cB)^* \subseteq \CS(\cA,\cB)$.
	
	The reverse inclusion does not hold in general.
	With $\cA=\set{0,2}$, $\beta=1$, $v_0=v_2=1$ and $v_1=-2$, we obtain $(\Pe{A,\beta})^* \subsetneq \Pe{A,\beta}$ as well as $(\Po{A,\beta})^* \subsetneq \Po{A,\beta}$.
	Setting $\cB=\set{1}$ it follows that $\CS(\cA,\cB)^* \subsetneq \CS(\cA,\cB)$.
\end{remark}

In addition, we need the following lemma for the proof of Theorem~\ref{thm:equiv}.
Here, for $\lambda \in \R_{+}^\cA$, denote by $\supp(\lambda) = \set{ \alpha \in \cA  \with \lambda_{\alpha} \neq 0}$
its support.

\begin{lemma}[Essentially Lemma 8 of \cite{mcw-2018}]
	\label{le:support}
	Let $\cA \subseteq\R^n$ be a non-empty finite set and $\beta \in \conv(\cA)$. 
	Then every $\lambda \in \Lambda(\cA, \beta)$ can be written as a sum
	\[ \lambda = \sum_{j=1}^k \mu_j \lambda^{(j)} \]
	with $k \geq 1$, $\mu \in \R^k_{+}, \sum_{j=1}^k \mu_j = 1$ and $\lambda^{(j)}\in \Lambda(\cA, \beta)$ for all $j$, such that the support of each $\lambda^{(j)}$ is affinely independent.
\end{lemma}
\begin{proof}
	Since the polytope $\Lambda(\cA, \beta)$ is the convex hull of its vertices, it suffices to show that the support of every vertex of $\Lambda(\cA, \beta)$ is an affinely independent set.
	
	Let $\lambda$ be a vertex of $\Lambda(\cA, \beta)$ and $\cA' := \set{\alpha \with \lambda_\alpha > 0}$ be its support.
	Assume to the contrary that $\cA'$ is affinely dependent.
	Then there exists $\mu \in \R^{\cA} \setminus \{0\}$ with $\sum_{\alpha \in \cA'} \mu_\alpha = 0$, $\sum_{\alpha \in \cA'} \mu_\alpha \alpha = 0$ and $\mu_{\alpha} = 0$ for $\alpha \not\in \cA'$.
	Since $\lambda_{\alpha} > 0$ for all $\alpha \in \cA'$, for sufficiently small $\epsilon > 0$ both $\lambda + \epsilon \mu$ and $\lambda - \epsilon \mu$ are contained in $\Lambda(\cA, \beta)$.
	But this implies that $\lambda = \frac{1}{2}(\lambda + \epsilon \mu) + \frac{1}{2}(\lambda - \epsilon \mu)$ is not a vertex of $\Lambda(\cA, \beta)$, a contradiction.
\end{proof}

\begin{proof}[Proof of Theorem \ref{thm:equiv}]
\newcommand{\tle}{\widetilde{\lambda}}
\newcommand{\tlz}{\widetilde{\lambda}'}
\begin{asparaenum}
	\item[(1):]
	Since $|\xb|^\alpha \in \CS(\cA,\cB)$ for every $\alpha \in \cA$, every $v \in \CS(\cA,\cB)^*$ satisfies
	\[ 0 \leq (v,w)(|\xb|^\alpha) = v_\alpha. \]
	
	\item[(2):]
	The implications (b) $\implies$ (c) $\implies$ (d) are trivial.
	For the equivalence of (a) and (b) note that
	\[
		\CS(\cA, \cB)^* = \bigcap_{\alpha \in\cA} (\Pe{\cA\setminus\set{\alpha}, \alpha})^* \cap \bigcap_{\beta \in\cB} (\Po{\cA, \beta})^*,
	\]
	because Minkowski sum and intersection are dual operations (see, e.g., \cite{schneider-book}, Theorem 1.6.3).
	Hence, the claim follows with Lemma~\ref{lem:dualcone}.
	It remains to show (c) $\implies$ (b) and (d) $\implies$ (c).
	
	\begin{asparaenum}
		\item[(c) $\implies$ (b):]
		Let $\beta \in \cA$ and $\lambda \in \Lambda(\cA, \beta)$.
		By Lemma~\ref{le:support}, we can decompose $\lambda$ as
		$\lambda = \sum_{j=1}^k \mu_j \lambda^{(j)}$ with $k \geq 1$,
		$\mu \in \R^k_{+}, \sum_{j=1}^k \mu_j = 1$ and $\lambda^{(1)}, \dotsc, \lambda^{(k)} \in \Lambda(\cA, \beta)$, such that the support of each $\lambda^{(j)}$ is affinely independent.
		Now the claim follows from
		\[
		\ln|v_\beta| = \sum_{j=1}^k \mu_j \ln|v_\beta|
		\stackrel{(c)}{\leq} \sum_{j=1}^k \mu_j \sum_{\alpha} \lambda^{(j)}_\alpha \ln v_\alpha
		= \sum_{\alpha}\ln v_\alpha \sum_{j=1}^k \mu_j  \lambda^{(j)}_\alpha 
		= \sum_{\alpha} \lambda_\alpha \ln v_\alpha.
		\]
		For $\beta \in \cB$, the proof is analogous by considering $w_\beta$ instead of $v_\beta$.
		
		\item[(d) $\implies$ (c):]
		We start with the even case and proceed by induction on $r = \re(A, \beta)$.
		Since the base case $r=0$ captures exactly the reduced circuits, there is nothing to prove in this case.
		
		Now consider an even circuit $(A, \beta) \in I(\cA, \cA)$ with $\re(A,\beta) > 0$.
		Then there exists a $\beta' \in \conv(A) \cap \cA$ with $\beta' \notin A$ and $\beta' \neq \beta$.
		Set $\lambda := \lambda(A,\beta)$ and $\lambda' := \lambda(A, \beta')$.
		
		Let $\tau \geq 0$ be the maximal real number with $\tle := \lambda - \tau \lambda' \in \R^A_{+}$.
		This number exists clearly, and we have $\tau \leq 1$ because the coordinate sums of $\lambda$ and $\lambda'$ are equal.
		Further, it holds that $\tau > 0$ because all components of $\lambda$ are positive.
		
		Similarly,
		let $\tau'$ be the maximal real number with $\tlz := \lambda' - \tau' \lambda \in \R^A_{+}$. As above, it holds that $0 \leq \tau' \leq 1$.
		Moreover, note that $\beta \neq \beta'$ implies $\tau, \tau' < 1$.
		The construction gives
		\[ \begin{aligned}
			\beta &= \sum_{\alpha\in A} \tle_\alpha \alpha + \tau \beta', & & & \sum_{\alpha\in A} \tle_\alpha + \tau &= 1,\\
			\beta' &= \sum_{\alpha\in A} \tlz_\alpha \alpha + \tau' \beta & &\text{and} & \sum_{\alpha\in A} \tlz_\alpha + \tau' &= 1.
		\end{aligned} \]
		Note that at least one of the entries of $\tle$ is zero, and moreover, 
		$\tau'$ or at least one of the entries of $\tlz$ is zero.
		Define two new even circuits $(A_1, \beta)$ and $(A_2, \beta')$ with $A_1 := \supp(\tle) \cup\set{\beta'}$ and
		\[
		A_2 := \begin{cases}
		\supp(\tlz) \cup \set{\beta} &\text{ if } \tau' > 0,\\
		\supp(\tlz) &\text{ if } \tau' = 0.
		\end{cases}
		\]
		We observe $\conv(A_1) \subsetneq \conv(A)$, and since $\beta'$ is not counted towards $\re(A_1,\beta)$, it follows that $\re(A_1,\beta) < \re(A, \beta)$.
		Similarly, since $\conv(A_2) \subseteq \conv(A)$ and $\beta'$ is not counted towards $\re(A_2, \beta')$, we obtain $\re(A_2,\beta') < \re(A, \beta)$.
		Hence, by induction,
		\begin{align}
			\ln(|v_\beta|) & \leq \sum_{\alpha\in A} \tle_\alpha \ln(v_{\alpha}) + \tau \ln(v_{\beta'}) \label{eq:circuit1} \qquad\text{and}\\
			\ln(|v_{\beta'}|) & \leq \sum_{\alpha\in A} \tlz_\alpha \ln(v_{\alpha}) + \tau' \ln(v_{\beta}). \label{eq:circuit2}
		\end{align}
		Note that $v_{\beta'} \ge 0$ and $v_{\beta} \ge 0$.
		Adding $\tau$ times~\eqref{eq:circuit2} to~\eqref{eq:circuit1}
		gives, due to $\tle + \tau \tlz = (1-\tau\tau')\lambda$,
		the uniform inequality
		\[
			0 \le (1-\tau \tau') \Big( \sum_{\alpha\in A} \lambda_\alpha \ln v_{\alpha} - \ln |v_\beta| \Big).
		\]
		Since $1-\tau \tau' > 0$, this proves the claim.

		For the odd case, we proceed by induction on $\ro(A, \beta)$, and the base case consists again of the reduced circuits.
		Fix an odd circuit $(A, \beta) \in I(\cA, \cB)$ with $\ro(A, \beta) > 0$.
		Again, there exists a $\beta' \in \conv(A) \cap \cA$ with $\beta' \notin A$, but this time $\beta' = \beta$ is possible.
		
		We define $\lambda, \lambda', \tau$ and $\tle = \lambda - \tau \lambda'$ as above. This time, $\tau = 1$ is possible.
		Further, we set $\tau' := 0$ and (thus) $\tlz := \lambda'$.
		We define the new circuits $(A_1, \beta)$ and $(A_2, \beta')$ as above, where this time $(A_1, \beta)$ is odd and $(A_2, \beta')$ is even.
		Since $\ro(A_1, \beta) < \ro(A, \beta)$ as above, we obtain
 		\begin{align}
 			\ln(|w_\beta|) & \leq \sum_{\alpha\in A} \tle_\alpha \ln(v_{\alpha}) + \tau \ln(v_{\beta'}) \label{eq:circuit3}\qquad\text{and}\\
 			\ln(|v_{\beta'}|) & \leq \sum_{\alpha\in A} \tlz_\alpha \ln(v_{\alpha}), \label{eq:circuit4}
 		\end{align}
 		where the second inequality follows since we have already shown (d) $\implies$ (c) for even circuits.
 		As above, we add $\tau$ times \eqref{eq:circuit4} to \eqref{eq:circuit3} to obtain the desired inequality.\qedhere
	\end{asparaenum} 
\end{asparaenum}
\end{proof}

A description of the dual of the SONC cone was obtained in \cite[Theorem 3.1]{dnt-2018},
and a description of the dual of the SAGE cone in \cite[Proposition 2.4]{chandrasekaran-shah-2016}.
Both descriptions are based on projections and differ from the one in Theorem~\ref{thm:equiv}.
For completeness, we show here that they are in fact equivalent.

\begin{proposition}\label{prop:dual}
	Let $\emptyset \neq \cA \subseteq\R^n$ be a finite set and $\beta \in \conv(\cA)$. 
	For $v \in \R_+^{\cA}$ and $w_\beta \in \R$, the following are equivalent:
	\begin{enumerate}[(1)]
		\item $\forall \lambda \in \Lambda(\cA, \beta) \colon \ln|w_\beta| \leq \sum_{\alpha\in \cA} \lambda_{\alpha} \ln(v_\alpha)$.
		\item $\exists\tau\in\R^n,\ \forall \alpha\in\cA \colon |w_\beta| \ln\left(\frac{|w_\beta|}{v_{\alpha }}\right) \leq (\beta - \alpha)^T \tau.$
		\item $\exists v^* \geq |w_\beta|, \exists\tau\in\R^n,\forall \alpha\in\cA \colon
		v^* \ln\left(\frac{v^*}{v_{\alpha}}\right) \leq (\beta - \alpha)^T \tau.$
	\end{enumerate}
\end{proposition}

In this proposition, statement (1) is the one we used earlier,
statement (2) is the description of the dual SAGE cone used in \cite{chandrasekaran-shah-2016},
and statement (3) in conjunction with Theorem~\ref{thm:equiv}(c)
is the description of the dual SONC cone used in \cite{dnt-2018}.
	
\begin{proof}
	If $w_\beta = 0$ then all three conditions hold.
	Moreover, if $v_{\alpha} = 0$ for some $\alpha\in\cA$, then it is easy to see that all three conditions hold if and only if $w_\beta = 0$.
	Thus we may assume that $w_\beta \neq 0$ and $v_{\alpha} \neq 0$ for all~$\alpha\in\cA$.
	We will show the equivalence via the following variant of statement (2),
	\begin{enumerate}[(1)]
		\item[(2')] $\exists\tau\in\R^n,\ \forall \alpha\in\cA \colon
					\ln\left(\frac{|w_\beta|}{v_{\alpha }}\right) \leq (\alpha - \beta)^T \tau.$
	\end{enumerate}
	\begin{asparaenum}
	\item[(1) $\iff$ (2'):]
		Consider (2') as the feasibility of a linear system of inequalities in $\tau$.
		(2') is satisfied if and only if its Farkas alternative system (in the version of Proposition 1.7 of \cite{Ziegler})
		\begin{equation*}
		\begin{aligned}
		\exists \lambda \in \R^{\cA}_{+}:\quad &\sum_{\alpha\in\cA} \lambda_\alpha(-\alpha + \beta) = 0 \ \text{ and }
		&\sum_{\alpha\in\cA} \lambda_\alpha \cdot \left(-\ln\left( \frac{|w_\beta|}{v_{\alpha}} \right)\right) < 0
		\end{aligned}
		\end{equation*}
		does not have a solution.
	
		We can normalize $\lambda$ so that all its components sum to 1.
		Hence, the alternative system simplifies to
		\[ \sum_{\alpha\in\cA} \lambda_\alpha \ln |w_\beta| > \sum_{\alpha\in\cA} \lambda_\alpha v_{\alpha} > 0, \]
		i.e., to $\ln |w_\beta| > \sum_{\alpha\in\cA} \lambda_\alpha v_{\alpha}$.
		Since this is the opposite of (1), the equivalence of (1) and (2') follows.
	\item[(2') $\implies$ (2):]
		We obtain (2) from (2') by multiplying with $|w_\beta|$ and replacing $|w_\beta|\tau$ by~$-\tau$.
	\item[(2) $\implies$ (3):]
		This is trivial.
	\item[(3) $\implies$ (2'):]
		We have that $v^* \geq |w_\beta| > 0$ and thus we may divide the inequality in (3) by $v^*$ to obtain
		\[
			\exists\tau'\in\R^n,\forall \alpha\in\cA \colon
			\ln\left(\frac{v^*}{v_{\alpha}}\right) \leq (\beta - \alpha)^T \tau',
		\]
		where $\tau' = \tau / v^*$.
		Note that the left-hand side of the inequality is monotonous in $v^*$, and hence,
		\[ \ln\left(\frac{|w_\beta|}{v_{\alpha}}\right) \leq  \ln\left(\frac{v^*}{v_{\alpha}}\right) \leq (\beta - \alpha)^T \tau'.\]
		We further replace $\tau'$ by $-\tau'$ to obtain (2').\qedhere
	\end{asparaenum}
\end{proof}

\section{Applications of the dual cone}

\subsection{Non-negative AG functions are sums of non-negative circuit functions}\label{sec:AGisscf}

As a first application of our description of the dual cone, we prove the following generalization of \cite[Theorem 4]{mcw-2018}.

\begin{proposition}\label{prop:AGisscf}
	Let $\emptyset \neq \cA \subseteq\R^n$ and $\cB \subseteq\N^n\setminus (2\N)^n$ be finite sets.
	For every $f\in\CS(\cA,\cB)$, the following statements hold.
	\begin{enumerate}
		\item $f$ can be written as a sum of non-negative circuit functions 
		  whose supports are contained in $\supp f$.
		\item $f$ can be written as a sum of non-negative circuit functions
		   supported on reduced circuits in $\CS(\cA,\cB)$.
	\end{enumerate}
\end{proposition}

Note that in statement (2), the support of the reduced circuits does not need to be contained in the support of $f$. The following example shows a situation, in which this phenomenon happens.

\begin{example} \label{ex:extremebsp1}
	Let $\cA:=\{0,2,4\}$ and $\cB:=\{1\}$.
	Consider the non-negative circuit function $f = |x|^0 - 4\cdot3^{-3/4} x + |x|^4=1-4\cdot3^{-3/4} x + x^4$.
	Its support $(\set{0,4},1)$ is not reduced with respect to $\cA, \cB$, and indeed, we can write $f$ as sum
	\[\begin{aligned}
	f = & \left( \frac{2}{3} - 4\cdot3^{-3/4} x + \frac{2}{3}\sqrt{3} x^2\right) + \left(\frac{1}{3} - \frac{2}{3}\sqrt{3} x^2 + x^4\right)\\
	= & \left( \frac{2}{3}|x|^0 - 4\cdot3^{-3/4} x + \frac{2}{3}\sqrt{3} |x|^2\right) + \left(\frac{1}{3}|x|^0 - \frac{2}{3}\sqrt{3} |x|^2 + |x|^4\right)
	\end{aligned}\]
	of non-negative circuit functions, whose supports $(\set{0,2},1)$ and $(\set{0,4},2)$ are reduced.
	Note that the coefficient of $|x|^2$ cancels in the sum.
\end{example}

\begin{proof} [Proof of Proposition~\ref{prop:AGisscf}]
	By Lemma~\ref{lem:dualcone} and part (c) of Theorem~\ref{thm:equiv}, the dual of the $\cS$-cone is 
	\begin{equation}\label{eq:dual-even-odd}
		\CS(\cA, \cB)^* = \bigcap_{(A,\beta) \in I(\cA, \cA)} (\Pe{A\setminus\{\beta\}, \beta})^* \cap \bigcap_{(A,\beta) \in I(\cA, \cB)} (\Po{A, \beta})^*.
	\end{equation}
	Let $f \in \CS(\cA, \cB)$ and assume that the support of $f$ is given by $\cA' \subseteq \cA$ and $\cB' \subseteq B$.
	By Proposition~\ref{pr:nocancel}, $f \in \CS(\cA',\cB')$.
	Apply~\eqref{eq:dual-even-odd} on the sub-cone $\CS(\cA', \cB')$ and dualize that identity.
	Using that $\CS(\cA', \cB')^{**} = \CS(\cA', \cB')$ (because the cone is closed, Proposition~\ref{prop:closed}) then yields 
	\[
		f \in \sum_{(A,\beta) \in I(\cA, \cA)} \Pe{A\setminus\{\beta\}, \beta} + \sum_{(A,\beta) \in I(\cA, \cB)} \Po{A, \beta}.
	\]
	This shows part (1).

	Part (2) then follows from part (d) of Theorem~\ref{thm:equiv}. Note that in this case we cannot restrict the sets of exponents to $\cA'$ and $\cB'$ as it depends on the choice of $\cA$ and $\cB$ whether a circuit is reduced or not.
\end{proof}

\begin{remark}
	If we demand $\supp(f)=\cA\cup\cB$, we obtain the same statement about the support in (2) as in (1).
\end{remark}

\subsection{Extreme rays of the \texorpdfstring{$\cS$}{S}-cone}\label{sec:extremerays}

Our next application of our description of the dual cone is a precise characterization of the extreme rays of $\CS(\cA,\cB)$.
Even for the specific case of the SAGE cone, this sharpens the result in \cite[Theorem 4]{mcw-2018}, where the necessary condition is that every extreme ray of the SAGE cone is supported on a single coordinate or on a circuit.
The essential concept for this characterization is provided by the reduced circuits.

Let $\emptyset \neq \cA\subseteq\R^n$ and $\cB\subseteq\N^n\setminus(2\N)^n$ be finite sets and write shortly $\lambda = \lambda(A,\beta)$.
For $(A,\beta)\in I(\cA,\cA)$ let
\[
	E_{\mathrm{e}}(A,\beta) :=
	\left\{ \sum_{\alpha\in A} c_\alpha|\xb|^\alpha - \prod\limits_{\alpha \in A} \left(\frac{c_\alpha}{\lambda_\alpha}\right)^{\lambda_\alpha}|\xb|^\beta
\with c \in \R_{>0}^A \right\},
\]
for $(A,\beta)\in I(\cA,\cB)$ let
\[
	E_{\mathrm{o}}(A,\beta) :=
	\left\{\sum_{\alpha\in A} c_\alpha|\xb|^\alpha \pm \prod_{\alpha \in A} \left(\frac{c_\alpha}{\lambda_\alpha}\right)^{\lambda_\alpha}\xb^\beta \with c\in\R^A_{>0}\right\},
\]
and for $\beta\in\cA$ let
\[
	E_1(\beta) :=
	\begin{cases}
		\R_{+}\cdot |\xb|^\beta  & \text{ if } \beta\in \cA\setminus\cB,\\
		\R_+\cdot(|\xb|^\beta \pm \xb^\beta) & \text{ if }\beta\in \cA\cap\cB.
	\end{cases}
\]
$E_\mathrm{e}(A,\beta)$ and $E_\mathrm{o}(A,\beta)$ are the (even and odd) non-negative
circuit functions, for which the inequality~\eqref{eq:circuitnumber} on the circuit number
holds with equality. $E_1(\beta)$ provides the special case for circuits supported on a single
element.

\begin{proposition}\label{prop:exray}
	For finite sets $\emptyset \neq \cA\subseteq\R^n$ and $\cB\subseteq\N^n\setminus(2\N)^n$,
	the set $\cE(\cA,\cB)$ of extreme rays of $\CS(\cA,\cB)$ is
		\begin{align*}
		\cE(\cA,\cB)=\left(\bigcup_{\substack{(A,\beta) \in I(\cA, \cA),\\r_{\mathrm{e}}(A,\beta)=0, |A|>1}} E_{\mathrm{e}}(A,\beta)\right)\cup \left(\bigcup_{\substack{(A,\beta) \in I(\cA, \cB),\\r_{\mathrm{o}}(A,\beta)=0, |A|>1}} E_{\mathrm{o}}(A,\beta)\right)\cup \left(\bigcup_{\beta\in\cA} E_1(\beta)\right).
		\end{align*}
\end{proposition}
Here, recall from Definition~\ref{de:reducedcircuit} that an even (respectively odd) circuit
is reduced if and only if $\re(A,\beta) = 0$ (respectively $\ro(A,\beta) = 0$).
The following example shows that the case distinctions are indeed necessary.

\begin{example} \label{ex:extreme1}
	For $\cA := \set{0,1,2}$ and $\cB := \set{1}$,
	the sets of (even resp. odd) circuits are  
	\begin{align*}
	I(\cA, \cA) &= \set{ (\set{0,2},1), (\set{0},0), (\set{1},1), (\set{2},2)} \text{ and }\\
	I(\cA, \cB) &= \{(\set{1},1),(\set{0,2},1)\}.
	\end{align*}
	We have a closer look at those elements which are both even and odd 
	circuits.
	\begin{itemize}
		\item[(1)] The circuit $(\set{0,2},1)$ is reduced as an even circuit and non-reduced as an odd circuit. In the context of extreme rays this is necessary. The even circuit function $a^2-2ab|x|+b^2x^2$ is an element of an extreme ray, but for the odd circuit function $a^2\pm2abx+b^2x^2$ we have
		\[ a^2\pm2abx+b^2x^2 = (a^2-2ab|x|+b^2x^2) + 2ab(|x|\pm x) \]
		and hence this is not an extreme ray of $\CS(\cA,\cB)$. 
		\item[(2)] Further, it holds that
		\[ |x| = \frac{1}{2}(|x|+x) + \frac{1}{2}(|x|-x), \]
		so $(\set{1}, 1)$ does not support an even extreme ray but in fact it does support an odd extreme ray.
	\end{itemize}
\end{example}

Again, we obtain corollaries for the special cases of 
SONC polynomials and of the
SAGE cone.

\begin{corollary}
	Let $\emptyset\ne \cA\subseteq \N^n$ be a finite set and write shortly $\lambda = \lambda(A,\beta)$. The set $\cE(\cA)$ of extreme rays of the cone of SONC-polynomials with support in $\cA$ is
	\begin{align*}
		\cE(\cA)= & \bigcup_{\substack{(A,\beta) \in I(\cA\cap (2\N)^n, \cA),\\r_{\mathrm{e}}(A,\beta)=0, \, |A|>1}} \left\{\sum_{\alpha\in A} c_\alpha x^\alpha - \prod_{\alpha \in A} \left(\frac{c_\alpha}{\lambda_\alpha}\right)^{\lambda_\alpha}x^\beta \ \vrule\  c\in\R^A_{>0}\right\} \\
	&	\cup   \bigcup_{\substack{(A,\beta) \in I(\cA\cap (2\N)^n, \cA),\\r_{\mathrm{o}}(A,\beta)=0, \, |A|>1, \, \beta\in\cA\setminus(2\N)^n }}
\left\{\sum_{\alpha\in A} c_\alpha x^\alpha + \prod_{\alpha \in A} \left(\frac{c_\alpha}{\lambda_\alpha}\right)^{\lambda_\alpha}x^\beta \ \vrule\  c\in\R^A_{>0} \right\}\\
	&	\cup  \bigcup_{\beta\in\cA\cap(2\N)^n} \R_{+}\cdot x^\beta.
	\end{align*}
\end{corollary}

\begin{corollary}
	Let $\emptyset \neq \cA \subseteq \R^n$ be a finite set, $\cB=\emptyset$ and write $\lambda = \lambda(A,\beta)$. The set $\cE(\cA,\cB)$ of extreme rays of the cone of SAGE functions with support in $\cA$ is 
	\begin{align*}
	\cE(\cA,\cB)= \ &  \bigcup_{\substack{(A,\beta) \in I(\cA, \cA),\\r_{\mathrm{e}}(A,\beta)=0, \, |A|>1}} E_{\mathrm{e}}(A,\beta) \ \cup \ \bigcup_{\beta\in\cA} E_1(\beta)\\
	= \ & \bigcup_{\substack{(A,\beta) \in I(\cA, \cA),\\r_{\mathrm{e}}(A,\beta)=0, \, |A|>1}} \left\{\sum\limits_{\alpha\in A} c_\alpha\exp(y^T\alpha) - \prod\limits_{\alpha \in A} \left(\frac{c_\alpha}{\lambda_\alpha}\right)^{\lambda_\alpha}\exp(y^T\beta) \with c \in \R^\cA_{>0}\right\}\\
	& \ \cup  \ \bigcup_{\beta\in\cA} \left\{c \exp(y^T\beta) \with c\in \R_+ \right\}.
	\end{align*}
\end{corollary}

\begin{example}
	As an example corresponding to the SAGE setting,
	let $\cA = \set{0,1,2,4}, \cB=\emptyset$ and $f:=|x|^0 - 4\cdot3^{-3/4} x^1 + |x|^4$ be a non-negative circuit function. With the substitution $x\mapsto \exp(y)$ we obtain the arithmetic-geometric exponential $\tilde{f} = 1 - 4\cdot3^{-3/4} \exp(y) + \exp(4y)$, and its support is again not reduced.
	We write $f$ as a sum
	\[ f = \left(1 - 2\cdot3^{1/4} |x| + \sqrt{3} x^2\right) + \left(\frac{2}{3}3^{1/4} |x| - \sqrt3x^2 + x^4 \right) \]
	of circuit functions, whose supports $\set{0,1,2}$ and $\set{1,2,4}$ are reduced.
	
	This is different from Example~\ref{ex:extremebsp1} in that the exponent $1$ is contained in $\cA$ rather than $\cB$ and thus is treated like an even number in the SAGE setting.
	
	In \cite{mcw-2018}, after Theorem 4, the authors remark that every
	circuit \enquote{supports a family of extreme rays in the SAGE cone.}
	This is not quite correct, as shown by the current example.
\end{example}

For the proof of Proposition~\ref{prop:exray}, we will use a variant of H\"older's inequality.

\begin{theorem}[Theorem 11, p. 22, \cite{HLP}]\label{thm:holder}
	Let $n, m \in \N$.
	Let $(a_{ij}) \in \R^{n\times m}$ be a matrix and let $\lambda_1, \dotsc, \lambda_{n} \in \R_{>0}$ with $\sum_{i=1}^n \lambda_{i} = 1$.
	Then 
	\[ 
	\sum_{j = 1}^{m} \prod_{i = 1}^{n} a_{i j}^{\lambda_{i}} 
	\leq \prod_{i = 1}^{n}\left(\sum_{j = 1}^{m} a_{i j} \right)^{\lambda_{i}},
	\]
	and equality holds if and only if either (1) for some $i$, $a_{i1} = \dotsb = a_{im} = 0$, or (2) the matrix  $(a_{ij})$ has rank one.
\end{theorem}
Note that in case (1) both sides of the inequality are zero.

\begin{proof}[Proof of Proposition~\ref{prop:exray}]
By Proposition~\ref{prop:AGisscf}, every non-negative function
$f \in \CS(\cA,\cB)$ can be written as a sum of non-negative circuit
functions supported on reduced circuits.
Hence, it suffices to show the following two statements:
	
\begin{enumerate}[(a)]
	\item Every non-negative circuit function supported on a circuit can be written as a sum of non-negative circuit functions with the same support whose circuit condition is satisfied with equality.
	\item Every function in $\cE(\cA,\cB)$ is indeed an extreme ray, i.e., it cannot be written as a sum of other non-negative AG functions.
\end{enumerate}
\begin{asparaenum}[(a)]
	\item Let $f$ be a non-negative circuit function supported on the circuit $(A, \beta)$, whose coefficients are denoted by $(c_\alpha)_{\alpha\in A}$ and $c_{\beta}$.

	If $f$ is supported on an odd circuit $(A, \beta) \in I(\cA, \cB)$, then the circuit number $\Theta_f$ from~\eqref{eq:circuitnumber} satisfies $c_\beta \in [-\Theta_f,\Theta_f]$.
	Hence, $f$ is a convex combination of $f_1$ and $f_2$, where $f_1,f_2$ have the same support and coefficients as $f$, except for $c_\beta^{(1)}=\Theta_{f_1}=\Theta_f$ and $c_\beta^{(2)}=-\Theta_{f_2}=-\Theta_f$.

	If $f$ is supported on an even circuit $(A, \beta) \in I(\cA, \cA)$, then $f$ can be written as the sum of a non-negative circuit function with the same support whose inner coefficient equals the negative of the circuit number and of some function $d|\xb|^\beta$ for $d>0$. 
	If $\beta\notin\cB$, then the latter is contained in $E_1(\beta)$.
	Otherwise, if $\beta\in\cB$, then $d|\xb|^\beta=\frac{d}{2}(|\xb|^\beta+\xb^\beta)+\frac{d}{2}(|\xb|^\beta-\xb^\beta) $, whose two summands are elements of $E_1(\beta)$.

\newcommand{\suppe}{\supp_{\mathrm{e}}}
	\item Let $f \in \cE(\cA, \cB)$ with coefficients $(c_\alpha)_{\alpha \in \cA}$ and $(d_\beta)_{\beta\in\cB}$. 
	Assume that $f$ can be decomposed into $f=\sum_{i=1}^{k}f_i$ with non-negative AG functions $f_1,\ldots,f_k \in \R[\cA, \cB]$.
	Denote the coefficients of $f_i$ by $(c_\alpha^{(i)})_{\alpha \in \cA}$ and $(d_\beta^{(i)})_{\beta\in\cB}$.

	For the duration of this proof, we use the notation $\suppe(f) := \set{\alpha\in\cA \with c_\alpha \neq 0}$.
	Moreover, set $\tilde{A} := \bigcup_i \suppe(f_i) = \set{\alpha\in\cA \with \exists i\colon c_\alpha^{(i)} \neq 0}$.
	We claim that $\tilde{A} \subseteq \conv \suppe(f)$.
	
	To show this, we consider a vertex $\tilde{\alpha}$ of $\conv\tilde{A}$.
	Since $\tilde{\alpha}$ must be an outer exponent of each $f_i$ with $c_{\tilde{\alpha}}^{(i)} \neq 0$, we have $c_{\tilde{\alpha}}^{(i)} \geq 0$ for all $i$.
	It follows that $\sum_{i=1}^k c_{\tilde{\alpha}}^{(i)} > 0$ and thus ${\tilde{\alpha}} \in \suppe(f)$.
	As this holds for every vertex of $\conv\tilde{A}$, we obtain that $\tilde{A}\subseteq \conv \suppe(f)$.

	Next, we distinguish three cases depending on whether $f \in E_1(\beta)$, $f \in E_{\mathrm{e}}(A, \beta)$ or $f \in E_{\mathrm{o}}(A, \beta)$.
\smallskip
	\begin{asparaitem}
	\item[\emph{Case $f \in E_1(\beta)$, $\beta \in \cA$}:]
		In this case, $\suppe(f_i) = \set{\beta}$ for each $i$.
		Thus, if $\beta \notin \cB$ then each $f_i$ is a multiple of $|\xb|^\beta$ and thus a multiple of $f$.
		
		On the other hand, if $\beta \in \cA \cap \cB$, then w.l.o.g. we can assume that $f = c(|\xb|^\beta + \xb^\beta)$ for some $c > 0$.
		Moreover, each $f_i$ is of the form $f_i = c_i |\xb|^\beta + d_i \xb^\beta$ with $|d_i| \leq c_i$.
		Then
		\[ \sum_i d_i = c = \sum_i c_i \geq \sum_i |d_i| \geq \Big|\sum_i d_i\Big| \]
		and thus $d_i = c_i$ for each $i$.
		Hence, all $f_i$ are multiples of $f$.

	\item[\emph{Case $f \in E_\mathrm{e}(A, \beta)$ for $(A, \beta) \in I(\cA, \cA)$ with $\re(A,\beta) = 0$ and $|A| > 1$}:]
	     In this case, our initial considerations imply that $\bigcup_i \suppe(f_i) \subseteq \conv(A)$.
		Since $(A, \beta)$ is reduced we can also conclude that $\bigcup_i \suppe(f_i) \subseteq A \cup\set{\beta}$.
		Hence, each $f_i$ is of the form
		\begin{equation}\label{eq:fi}
			f_i = \sum_{\alpha \in A} c^{(i)}_\alpha |\xb|^\alpha + c^{(i)}_\beta |\xb|^\beta + \sum_{\beta' \in \cB} d^{(i)}_{\beta'} \xb^\beta.
		\end{equation}
		It follows that $c^{(i)}_\alpha \geq 0$ for all $i$ and $\alpha \in A$, because otherwise the $f_i$ cannot be non-negative.
		
	Next, we claim that
	\begin{equation}\label{eq:aux}
	 -c^{(i)}_\beta \leq \prod_{\alpha\in A} \left(\frac{c_\alpha^{(i)}}{\lambda_\alpha}\right)^{\lambda_\alpha} 
	\end{equation}
	for all $i$, where again we write $\lambda = \lambda(A, \beta)$.
	To prove the claim, we distinguish two cases:
	\begin{enumerate}[(i)]
	\item If $c^{(i)}_\beta \geq 0$, then it trivially holds that 
	$-c^{(i)}_\beta \leq 0 \leq \prod_{\alpha\in A} ({c_\alpha^{(i)}}/{\lambda_\alpha})^{\lambda_\alpha}$.
	\item Consider the case that $c^{(i)}_\beta < 0$.
		Since $f_i$ is a non-negative AG function, it holds that
		the last sum in \eqref{eq:fi} vanishes and the claim follows from Theorem~\ref{thm:oddImplication}(3).
	\end{enumerate}

	In the next step, we derive
		\begin{equation}\label{eq:holder}
			-c_\beta
			= - \sum_{i =1}^k c_\beta^{(i)}
			\overset{\text{(a)}}{\leq} \sum_{i =1}^k \prod_{\alpha\in A} \left(\frac{c_\alpha^{(i)}}{\lambda_\alpha}\right)^{\lambda_\alpha}
			\overset{\mathrm{(b)}}{\leq}\prod_{\alpha\in A} \left(\sum_{i=1}^{k}\frac{c_\alpha^{(i)}}{\lambda_\alpha}\right)^{\lambda_\alpha}
			\overset{\mathrm{(c)}}{=} 
			 \prod_{\alpha\in A} \left(\frac{c_\alpha}{\lambda_\alpha}\right)^{\lambda_\alpha}
			= -c_\beta,
		\end{equation}
		
		where in (a) we use \eqref{eq:aux}, (b) follows from H\"older's Inequality~\ref{thm:holder} and (c) uses that $\sum_{i=1}^k c_{\alpha}^{(i)} = c_{\alpha}$.
		Moreover, by Theorem~\ref{thm:holder} equality in (b) implies that either (1) there exists an $\alpha \in A$ such that $c_\alpha^{(i)}$ vanishes for all $i$,
		or (2) the $|A| \times k$ matrix with entries $c^{(i)}_\alpha / \lambda_\alpha$ has rank one.
		However, (1) would imply that $c_\beta = 0$ which is impossible, thus we are in case (2).
		Hence, there exist scalars $\epsilon_1, \dotsc, \epsilon_k \geq 0$ such that $c^{(i)}_\alpha = \epsilon_i c_\alpha$ for all $i$ and all $\alpha \in A$.
		Further, equality in (a) implies that
		\[
		 -c^{(i)}_\beta = 
		\prod_{\alpha\in A} \left(\frac{c_\alpha^{(i)}}{\lambda_\alpha}\right)^{\lambda_\alpha}
		= \prod_{\alpha\in A} \left(\epsilon_i\frac{c_\alpha}{\lambda_\alpha}\right)^{\lambda_\alpha}
		= - \epsilon_i c_\beta.
		 \]
		By \eqref{eq:fi}, it follows that every $f_i$ is of the form
		\begin{equation}\label{eq:epsilon} 
		f_i = \epsilon_i f + \text{terms in }\cB. 
		\end{equation}
		Now, if $\epsilon_i = 0$ for some $i$, then $f_i$ has only terms with exponents in $\cB$ and thus it is the zero function or it cannot be non-negative. It follows that $\epsilon_i > 0$ for all $i$.
		But this implies that $c^{(i)}_\beta = \epsilon_i c_\beta < 0$.
		Hence, since the $f_i$ are AG functions, they cannot have any other terms in $\cB$, and thus they are all multiples of $f$.

	\item[\emph{Case $f \in E_\mathrm{o}(A, \beta)$ for $(A, \beta) \in I(\cA, \cB)$ with $\ro(A,\beta) = 0$ and $|A| > 1$}:] 
	In this case, the argument is similar, except that \eqref{eq:holder} becomes
		\begin{equation}\label{eq:holder2}
			|d_\beta|
			= \left|\sum_{i =1}^k d_\beta^{(i)}\right|
			\overset{\text{(e)}}{\leq} \sum_{i =1}^k |d_\beta^{(i)}|
			\leq \sum_{i =1}^k \prod_{\alpha\in A} \left(\frac{c_\alpha^{(i)}}{\lambda_\alpha}\right)^{\lambda_\alpha}
			\leq\prod_{\alpha\in A} \left(\sum_{i=1}^{k}\frac{c_\alpha^{(i)}}{\lambda_\alpha}\right)^{\lambda_\alpha}
			=
	 \prod_{\alpha\in A} \left(\frac{c_\alpha}{\lambda_\alpha}\right)^{\lambda_\alpha}
			= |d_\beta|.
			\end{equation}
		Since we have equality in (e), it follows that all terms on the left-hand side of that triangle inequality have the same sign.
		Since $d_\beta = \sum_{i=1}^k d^{(i)}_\beta$, this implies that each $d^{(i)}_\beta$ has the same sign as $d_\beta$.
		Now we also obtain \eqref{eq:epsilon}.
		Note that for the vanishing of the terms with exponents in $\cB\setminus\set{\beta}$, we can argue as above, or alternatively obtain this directly from the oddness of the $f_i$.
		Altogether, this yields again that all the $f_i$ are multiples of $f$.
		\qedhere
	\end{asparaitem}
	\end{asparaenum}
\end{proof}

\subsection{Univariate SONC polynomials do not satisfy Putinar's Positivstellensatz}\label{se:putinar}

In \cite[Section 5]{DKdW}, a multivariate example was given to show that the analogue of Putinar's Positivstellensatz does not hold for SONC polynomials. Using Theorem~\ref{thm:equiv}, we provide a simpler example of this phenomenon, which in addition shows that the analogue of Putinar's result does not even hold for univariate SONC polynomials.

Recall Putinar's Theorem from the theory of sums of squares polynomials (\cite{putinar-1993}, see also, e.g., \cite[Theorem 2.14]{lasserre-positive-polynomials}).

\begin{theorem}
	Let $f, g_1, \ldots, g_m \in \R[\xb]$ and assume that the quadratic module
	\[
		Q(g_1, \ldots, g_m) := \left\{ p_0 + \sum_{j=1}^m p_j g_j
		\text{ with sums of squares polynomials } p_0, \ldots, p_m \right\}
	\]
	is Archimedean.
	If $f \in \R[\xb]$ is strictly positive on the set $K = \{\xb \in \R^n \, : \, g_j(\xb) \ge 0, \, 1 \le j \le m\}$, then $f$ can be written in the form $f = p_0 + \sum_{j=1}^m p_j g_j$ with sum of squares polynomials $p_0, \ldots, p_m$.
\end{theorem}

Here, the Archimedean condition can be defined by the existence of some $N \ge 1$ with $N-\sum_{i=1}^n x_i^2 \in Q(g_1, \ldots, g_m)$, and it is well known that $Q(g_1, \ldots, g_m)$ is Archimedean if $g_1, \ldots, g_m$ are affine (see, e.g., \cite{lasserre-positive-polynomials}).

\begin{theorem}\label{thm:noputinar}
	The univariate polynomial
	\[ f := \left(x - \frac{1}{2}\right)^4 + \frac{1}{1000} \]
	satisfies $f(x) > 0$ for $x \in [0,1]$, but it cannot be written in the form 
	\begin{equation} \label{eq:putinar1}
		p_0 + x p_1 + (1-x) p_2 + x(1-x) p_3 
	\end{equation}
	with SONC polynomials $p_0, p_1, p_2$ and $p_3$.
\end{theorem}
Note that a SONC analogue of Putinar's Positivstellensatz would even assert a representation of $f$ using only $p_0, p_1$ and $p_2$.

\begin{proof}
	As a notation, for $r \in \N$ we set $\CS(r) := \CS(\set{0,1,\dotsc, r} \cap 2\N, \set{0,1,\dotsc, r} \setminus 2\N)$. This is the cone of SONC polynomials of degree up to $r$.

	The polynomial $f$ is clearly positive on $[0,1]$ (in fact, on $\R$), so we only need to show that it does not have a representation as in the claim.
	
	Assume to the contrary that there exist SONC polynomials $p_0, p_1, p_2$, 
	$p_3$ such that \eqref{eq:putinar1} holds.
	Let $d$ be the maximum of the degrees of the $p_i$.
	Then $f$ is contained in $\CS(d+2) + x \CS(d+1) + (1-x)\CS(d+1) + x(1-x)\CS(d)$.
	On the other hand, consider the vector
	\[
	v := \left(\frac{25}{18}, \frac{5}{9}, \frac{1}{2^2}, \frac{1}{2^3}, \dotsc, \frac{1}{2^{d+2}}\right)\in \R^{\set{0,\dotsc,d+2}}.
	\]
	A direct computation shows that
	\[
	v(f) = -\frac{1}{288} + \frac{25}{18}\frac{1}{1000} = -\frac{1}{480} < 0.
	\]
	Hence, once we show that $v$ lies in
	\begin{multline*}
	\Big(\CS(d+2) + x\CS(d+1) + (1-x)\CS(d+1) + x(1-x)\CS(d)\Big)^* \\
	\qquad = \CS(d+2)^* \cap (x\CS(d+1))^* \cap ((1-x)\CS(d+1))^* \cap (x(1-x)\CS(d))^*,
	\end{multline*}
	we obtain a contradiction.
	
	For this, note that we only need to consider inequalities involving the first two components of $v$, because apart from those $v$ equals the vector 
	$((\frac{1}{2})^{\alpha})_{0 \leq \alpha \leq d+2}$, which is clearly contained in the cone.
	Further, note that $v \in (x \CS(d+1))^*$ if and only if the shifted vector $(v_1, v_2, \dotsc, v_{d+2})\in \R^{\set{0,\dotsc,d+1}}$ where we omitted the $0$-th coordinate lies in $\CS(d+1)^*$.
	Similarly, $v$ lies in $((1-x)\CS(d+1))^*$ if and only if $(v_0-v_1, v_1-v_2, \dotsc, v_{d+1}-v_{d+2})$ lies in $\CS(d+1)^*$,
	and an analogous description holds for $(x(1-x) \CS(d))^*$.
	Using these observations, it is a straightforward computation to verify that $v$ lies in the cone.
\end{proof}

\subsection{Functions with simplex Newton polytopes}

We provide a subclass of functions for which non-negativity coincides
with containment in the $\cS$-cone $\CS(\cA,\cB)$.

\begin{proposition}\label{prop:simplex}
	Let $\emptyset \neq \cA \subseteq \R^n$, $\cB \subseteq \N^n \setminus(2\N)^n$ be finite sets and $f = \sum_{\alpha\in\cA} c_\alpha|\xb|^\alpha + \sum_{\beta\in\cB}d_\beta\xb^\beta \in\R[\cA, \cB]$.
	Assume that
	\begin{enumerate}
		\item $\conv(\cA)$ is a simplex and $\cB \subseteq \conv(\cA)$,
		\item $c_\alpha \leq 0$ for every $\alpha \in \cA$ which is not a vertex of $\conv(\cA)$, and
		\item $d_\beta \leq 0$ for every $\beta \in \cB$.
	\end{enumerate}
	Then $f$ is non-negative if and only if $f \in \CS(\cA,\cB)$.
	In this case, $f$ can be written as a sum of circuit functions using only vertices of $\conv(\cA)$ as outer exponents.
\end{proposition}

This has been shown for SONC polynomials under a slightly stronger hypothesis in \cite[Theorem 5.5]{iliman-dewolff-resmathsci}.
Moreover, the analogous statement in the SAGE setting has been obtained in \cite[Theorem 10]{mcw-2018}.
We provide a simple proof using Theorem~\ref{thm:equiv} as well as the following lemma.

\begin{lemma}\label{lem:simplex}
	Let $\emptyset \neq A \subseteq \R^n$ be affinely independent and let $v \in \R^A_{> 0}$.
	Then there exists a point $\pb \in \R^n_{>0}$ and a scalar $\tau \in \R_{>0}$ such that $v_\alpha = \tau \pb^\alpha$ for all $\alpha \in A$.
\end{lemma}
\begin{proof}
	Since $A$ is affinely independent, there exists an affine map $\ell: \R^n \to \R$ such that $\ell(\alpha) = \ln(v_\alpha)$ for all $\alpha \in A$.
	Explicitly, there exists a $\wb \in \R^n$ and $c \in \R$ with $\wb^T\alpha + c = \ln(v_\alpha)$ for all $\alpha \in A$.
	We set $\tau := \exp(c)$ and $p_i := \exp(w_i)$ for $1 \leq i \leq n$.
	A straightforward computation shows that $\pb := (p_1, \dotsc, p_n)$ and $\tau$ satisfy our claim:
	\[
		v_\alpha = \exp( \ell(\alpha)) = \exp\left( \sum_{i} \alpha_i w_i + c \right) = \exp(c) \prod_i \exp(w_i)^{\alpha_i} = \tau \prod_i p_i^{\alpha_i} = \tau \pb^\alpha.
	\qedhere
	\]
\end{proof}

\begin{proof}[Proof of Proposition~\ref{prop:simplex}]
        For the nontrivial direction,
	let $f$ be non-negative and denote by $V \subseteq \cA$ the set of vertices of $\conv(\cA)$. We show that $f$ is contained in the sub-cone 
	\[ C := \sum_{\alpha \in \cA \setminus V} \Pe{V, \alpha} + \sum_{\beta \in \cB} \Po{V, \beta} \subseteq \CS(\cA,\cB).
	\]
	Let 
	$(v,w)$ be an arbitrary element of $C^* = \bigcap_{\alpha \in \cA \setminus V} (\Pe{V, \alpha})^* \cap \bigcap_{\beta \in \cB} (\Po{V, \beta})^*$.
	First, consider the case that $v_\alpha > 0$ for all $\alpha \in V$.
	Since $V$ is affinely independent, Lemma~\ref{lem:simplex} gives a $\pb \in \R^n_{>0}$ and a $\tau \in \R_{> 0}$ with $v_\alpha = \tau \pb^\alpha$ for all $\alpha \in V$.
	For $\beta \in (\cA \cup \cB) \setminus V$ set $\lambda^{(\beta)} := \lambda(V, \beta)$ and observe that by Lemma~\ref{lem:dualcone},
	\[
		|v_\beta| \leq \prod_{\alpha \in V} v_\alpha^{\lambda^{(\beta)}_\alpha} = \tau \pb^{\sum_{\alpha\in V} \alpha \lambda^{(\beta)}_\alpha} = \tau \pb^\beta, 
	\]
	respectively $|w_\beta| \leq \tau \pb^\beta$.
	Hence,
	\[
	\begin{aligned}
	(v,w)(f) &= \sum_{\alpha\in V} v_\alpha c_\alpha + \sum_{\alpha \in \cA \setminus V} v_\alpha c_\alpha + \sum_{\beta \in \cB} w_\beta d_\beta \\
	&\geq \sum_{\alpha\in V} v_\alpha c_\alpha - \sum_{\alpha \in \cA \setminus V} |v_\alpha c_\alpha| - \sum_{\beta \in \cB} |w_\beta d_\beta| \\
	&\geq \sum_{\alpha\in V} \tau \pb^\alpha c_\alpha - \sum_{\alpha \in \cA \setminus V} |\tau \pb^\alpha c_\alpha| - \sum_{\beta \in \cB} |\tau \pb^\beta d_\beta|. 
	\end{aligned}
	\]
	Therefore, the hypotheses (2) and (3) imply that
	\[
	  (v,w)(f) \geq \tau f(\pb) \geq 0.
	\]
	In the case $v_{\alpha} = 0$ for some $\alpha \in V$, continuity of 
	the mapping in~\eqref{eq:pairing} implies $(v,w)(f) \ge 0$ as well.
	Altogether, $f \in C^{**} = C \subseteq \CS(\cA,\cB)$.
\end{proof}

\subsection{Approximating non-negative polynomials by SONC polynomials}

Unlike the situation with sum of squares polynomials, 
not every non-negative univariate polynomial is a SONC polynomial.
However, in this section we show that non-negative univariate polynomials 
can at least be approximated by SONC polynomials.

For a univariate polynomial $f = \sum_{i=0}^{d} c_i x^i$ we set
\[ \hat{f} := c_0 - \sum_{i=1}^{d-1} |c_i| x^i + c_d x^d. \]
This is very similar to the SAGE-representative of a polynomial considered in \cite[Section~5]{mcw-2018}.

\begin{theorem}\label{thm:approx2}
	Let $f = \sum_{i=0}^d c_i x^i \in \R[x]$ be a univariate polynomial of even degree $d$ with positive constant coefficient.
	Let $x_0 := \inf\{x \in \R_+ \with \hat{f}(x) < 0\}$.
	\begin{enumerate}
		\item If $x_0 = +\infty$ (i.e., if $\hat{f}(x) \geq 0$ for all $x \in \R$), then $f$ is a SONC polynomial.
		\item Otherwise, there exists a sequence $(p_N)_N \subseteq \R[x]$ of SONC polynomials which converges to $f$ uniformly on every compact subset of the open interval $(-x_0, x_0)$.
	\end{enumerate}
\end{theorem}
Note that the $p_N$ have the property that $\deg p_N \to \infty$ for $N \to \infty$.
Part (1) is essentially a special case of Proposition~\ref{prop:simplex}, which itself has been obtained before by
 de Wolff and Iliman \cite[Theorem 5.5]{iliman-dewolff-resmathsci}, see also \cite[Theorem 10]{mcw-2018}. Hence, only part (2) is new. It 
can be seen as a univariate SONC analogon of
the (even multivariate) approximation result in terms of sum of squares polynomials
by Lasserre and Netzer (see \cite{lasserre-netzer-2007,lasserre-2007}),
where our result also has a restriction to $(-x_0,x_0)$.

\begin{proof}
	For part (1), note that $\hat{f}$ is non-negative on $\R_+$ if and only if it is non-negative on $\R$. Moreover, it satisfies the hypothesis of Proposition~\ref{prop:simplex}, and thus $x_0 = \infty$ implies that $\hat{f}$ is a SONC polynomial.
	Moreover, Proposition~\ref{prop:simplex} implies that $\hat{f}$ can be written as a sum of circuit polynomials using only the constant term and the highest term as outer exponents.
	Since a non-negative circuit polynomial with negative inner coefficient remains non-negative under flipping the sign of the inner coefficient, $f$ is a SONC polynomial as well.

	It remains to show part (2). As in the proof of Theorem~\ref{thm:noputinar}, we use the shorthand notation $\CS(r)$ for $\CS(\set{0,1, \dotsc, r} \cap 2\N, \set{0,1, \dotsc, r} \setminus 2\N)$ for $r \in \N$. Since the two parameters of $\CS$ are disjoint sets, we shortly write elements in the dual cone as $v$ rather than $(v,w)$, by slight abuse of notation.
	For $N > d$ define
	\[
	p_N := \begin{cases}
		f + \frac{c^*}{x_0^N} x^{N} & \text{ for $N > d$ even,} \\
		p_{N-1}                     & \text{ for $N > d$ odd}
	\end{cases}
	\]
	for some constant $c^* > 0$.
	It is immediately clear that if $|x| < x_0$, then $f(x) - p_N(x) = c^* (x / x_0)^N$ converges to zero for $N \to \infty$, and we even have uniform 
	convergence on every compact subset of $(-x_0,x_0)$.
	It remains to find a suitable value of $c^*$ such that $p_N$ is 
	a SONC polynomial.
	
	We claim that $v(f) \ge m (v_0 + x_0^{-d}v_{d})$ for every 
	$v \in \CS(d)^*$, where $m$ is the
	minimum of the function $v \mapsto v(f)$ on the set
	\[ K := \left\{ \frac{v}{v_0 + x_0^{-d} v_d} \with 
	  v \in \CS(d)^* \setminus \{0\} \right\}
	=  \left \{ v \in \CS(d)^* \with v_0 + x_0^{-d}v_{d} = 1 \right\}. 
	\]
	
	To prove the claim, consider a fixed $v \in K$ and observe that 
	$v_{d} \ge 0$ because $d$ is even. The definitions of $K$ and $x_0$
	imply $v_0 \le 1$ and $v_{d} \leq x_0^{d}$.
	For $1 \le i < d$, Theorem~\ref{thm:equiv}(c) 
	applied on the circuit with outer
	exponents $0$, $d$ and inner exponent $i$ then gives
	\[
	  |v_i| \leq v_0^{1-i/d} v_{d}^{i/d}
	  \le 1 \cdot (x_0^{d})^{i/d} \le x_0^i.
	\]
	Hence, the non-empty set $K$ is bounded and compact.
	It follows that the function $v \mapsto v(f)$ attains the minimum value $m$
	on $K$. This implies the claim.

	Since for $N>d$ every element of $\CS(N)^*$ can be truncated to obtain an element of $\CS(d)^*$, it follows that the auxiliary claim also holds for every $v \in \CS(N)^*$.
	We set $c^* := 2|m|$ and it remains to show that for this value of $c^*$ our candidate $p_N$ is a SONC polynomial.
	For this, it is sufficient to show that for $N > d$ we have 
	$v(p_N) \geq 0$ for all $v \in \CS(N)^*$, and we may assume that $N$ is even.
	We distinguish two cases.
	\begin{asparadesc}
		\item[\normalfont{\emph{Case $v_{d} \leq x_0^{d} v_0$.}}]
		The inequalities defining $\CS(N)^*$ imply that $|v_i| \leq v_0^{1-i/d}v_{d}^{i/d}$ for $0 \leq i \leq d$.
		Using this, we derive
		\begin{align*}
		v(p_N) &= \sum_{i=0}^{d} v_i c_i + v_N \frac{c^*}{x_0^N}
		\geq v_0 c_0 - \sum_{i=1}^{d-1} |v_i c_i| + v_{d} c_{d} + v_N \frac{c^*}{x_0^N} \\
		&\geq v_0 c_0 - \sum_{i=1}^{d-1} v_0^{1-i/d}v_{d}^{i/d} |c_i| + v_{d} c_{d} + v_N\frac{c^*}{x_0^N}.
		\end{align*}
		If $v_0 = 0$, then this expression is non-negative since both $v_{d}$  and $v_N$ are.
		Otherwise, we continue as follows:
		\begin{align*}
		v(p_N) &= v_0\left(c_0 - \sum_{i=1}^{d-1}\left(\frac{v_{d}}{v_0}\right)^{i/d} |c_i| +  \left(\frac{v_{d}}{v_0}\right)^{d/d} c_{d}\right) + v_N\frac{c^*}{x_0^N}\\
		&= v_0 \hat{f}\left(\left(\frac{v_{d}}{v_0}\right)^{1/{d}}\right) + v_N \frac{c^*}{x_0^N} \geq 0
		\end{align*}
		for any $c^* \geq 0$ by the choice of $x_0$.
		\item[\normalfont{\emph{Case $v_{d} > x_0^{d} v_0$.}}]
		By Theorem~\ref{thm:equiv}(c) applied on the circuit with
		outer exponents $0$, $N$ and inner exponent $d$, we have
		$v_N \geq v_0^{-(N - d)/d} v_{d}^{N / d}$, so that
		using the hypothesis of the current case twice gives
		\begin{equation}\label{eq:case2}
			v_N > (x_0^{-d} v_{d})^{-(N - d)/d} v_{d}^{N / d} 
			= x_0^{N-d} v_{d} > \frac{x_0^N}{2}(v_0 + x_0^{-d}v_{d}).
		\end{equation}
		Since $v(p_N) = v(f) + v_N \frac{c^*}{x_0^{N}}$, 
		employing the auxiliary claim as well as~\eqref{eq:case2}
		we can conclude
		\begin{align*}
			v(p_N) &\geq (v_0+x_0^{-d}v_{d})m + \frac{x_0^N}{2}(v_0+x_0^{-d}v_{d}) \frac{c^*}{x_0^{N}} \\
			&= (v_0+x_0^{-d}v_{d})\cdot(m + |m|) \ge 0. \qedhere 
		\end{align*}
	\end{asparadesc}
\end{proof}

As a corollary of the theorem, we see that we can also approximate in the $(x)$-adic topology:
\begin{corollary}
	Let $f \in \R[x]$ be a univariate polynomial with $f(0) > 0$.
	Then for each $N \geq 0$ there exists a SONC polynomial $p_N \in \R[x]$ such that
	\[ f \equiv p_N \mod{x^N}. \]
\end{corollary}
\begin{proof}
	If the degree of $f$ is odd, then we may consider $f$ as a polynomial of higher degree with leading coefficient $0$, which has even degree.
	The hypothesis $f(0) > 0$ implies that the $x_0$ of Theorem~\ref{thm:approx2} exists and is positive, hence we may consider the sequence $p_N$ from that theorem.
	From its construction in the proof of Theorem~\ref{thm:approx2}, it is clear that it satisfies our claim.
\end{proof}

\section{Outlook and open problems}

 We have introduced the $\cS$-cone as a unified framework for
  the classes of SAGE and SONC polynomials, provided characterizations
  of its dual cone and presented several new
  and several improved results associated with the dual viewpoint.
  The $\cS$-cone exhibits a prominent computationally tractable 
  class within the class of sparse non-negative polynomials.
  For further computational aspects building upon the projection-free
  descriptions of the dual cones from Section~\ref{se:circuits-dual},
  we refer to the subsequent work of the second author together with 
  Dressler, Heuer and de Wolff \cite{hnw-2020}.

  It remains a future task to further understand the
  relation of the $\cS$-cone and its specializations
  to the underlying class of all non-negative functions (in some special
  cases polynomials), both from the primal
  and the dual point of view. Specifically, the relation of the 
  SONC cone to the cone of sparse non-negative polynomials 
  and the dual SONC cone to sparse moment cones (as studied
  by Nie \cite{nie-2014}) deserve further study. 
  It is an open question whether SONC polynomials 
  are dense inside the non-negative ones.

  Moreover, since by the results in Section~\ref{se:putinar}, 
  the analogue of Putinar's
  Positivstellensatz already fails in the univariate case, it also
  remains a challenge to provide computationally attractive
  types of Positivstellens\"atze for the $\cS$-cone and its
  specializations.

\medskip

\noindent
{\bf Acknowledgment.} We thank the anonymous referees for their helpful
suggestions.

\iftoggle{usebiblatex}{%
\printbibliography%
}{
\bibliographystyle{amsplain}
\bibliography{bibSoncSage}
}

\end{document}